\documentclass[11pt,reqno]{amsart}
\usepackage{ytableau, tikz,hyperref}
\usepackage[backend=bibtex, style=alphabetic]{biblatex}
\usepackage{palatino}
\usepackage[left=3cm,right=3cm,top=3cm,bottom=3cm]{geometry}
\newcommand\characx{\mathbf{x}}
\newcommand{\key}{\kappa}
\newtheorem{theorem}{Theorem}
\newtheorem{remark}{Remark}
\newtheorem{proposition}{Proposition}
\newtheorem{corollary}{Corollary}
\newtheorem{lemma}{Lemma}
\newtheorem{definition}{Definition}
\DeclareMathOperator{\Tab}{Tab}
\DeclareMathOperator{\SHive}{SHive}
\DeclareMathOperator{\Hive}{Hive}
\DeclareMathOperator{\GT}{GT}
\allowdisplaybreaks
\title[Saturation for Flagged Skew Littlewood-Richardson Coefficients]{Saturation for Flagged Skew\\ Littlewood-Richardson coefficients}
\author[]{Siddheswar Kundu}
\address{The Institute of Mathematical Sciences, A CI of Homi Bhabha National Institute, Chennai 600113, India}
\email{siddheswark@imsc.res.in}
\author[]{K.N. Raghavan}
\address{The Institute of Mathematical Sciences, A CI of Homi Bhabha National Institute, Chennai 600113, India}
\email{knr@imsc.res.in}
\author[]{V. Sathish Kumar}
\address{The Institute of Mathematical Sciences, A CI of Homi Bhabha National Institute, Chennai 600113, India}
\email{vsathish@imsc.res.in}
\author[]{Sankaran Viswanath}
\address{The Institute of Mathematical Sciences, A CI of Homi Bhabha National Institute, Chennai 600113, India}
\email{svis@imsc.res.in}
\date{}
\keywords{Skew Hives, Skew GT Patterns, Saturation, Flagged Littlewood-Richardson Coefficients, Crystals}  
\thanks{The authors acknowledge partial funding from a DAE Apex Project grant to the Institute of Mathematical Sciences, Chennai.}
\subjclass[]{05E05 (05E16)}
\begin{document} 
\begin{abstract}
We define and study a generalization of the Littlewood-Richardson (LR) coefficients, which we call the flagged skew LR coefficients. These subsume several previously studied extensions of the LR coefficients. We establish the saturation property for these coefficients, generalizing work of Knutson-Tao and Kushwaha-Raghavan-Viswanath. 
\end{abstract}

\maketitle
\section{Introduction}

The Littlewood-Richardson (LR) coefficients are among the most celebrated numbers in algebraic combinatorics. They are the multiplicities in the decomposition into irreducibles of tensor products of irreducible polynomial representations of general linear groups. As such, they are the structure constants for the multiplication of Schur polynomials (which form a basis for the ring of symmetric polynomials).
These coefficients also determine the branching of irreducible representations of symmetric groups on restriction to Young subgroups.
They also come up in several other places.
For example, they occur as structure constants for the multiplication of Schubert cohomology classes of Grassmanians. 

Several generalizations of these coefficients can be found in the literature: e.g., \cite{Z}, \cite{RS}, \cite{krvfpsac}. Our broader goal is to investigate, for some of these generalizations, the analogue of the saturation theorem of Knutson-Tao for the LR coefficients. In this paper, we consider a simultaneous generalization of Zelevinsky's skew LR coefficients \cite{Z} and the flagged LR coefficients of \cite{krvfpsac}. We prove that these {\em flagged skew LR coefficients} exhibit the saturation property.

In order to do this, we need to lift the main {\em key-positivity} result of \cite{RS} to crystals; we do this using a recent result of Assaf \cite{sami}. We also give a hive-like model for the flagged skew LR coefficients. Although a skew Gelfand-Tsetlin (GT) polytope need not be isomorphic to a GT polytope, it turns out that every flagged skew hive polytope is isomorphic to some flagged hive polytope.

\section{Preliminaries and Statements of Main Theorems}\label{s:prelims}
\quad A \emph{partition} $\lambda = (\lambda_1, \lambda_2, \cdots)$ is a weakly decreasing sequence of non-negative integers with finitely many non-zero terms (or parts). The \emph{length} of the partition $\lambda$ is defined to be the largest integer $i$ such that $\lambda_i$ is non-zero and we denote it by $l(\lambda)$. The \emph{weight} of $\lambda $ is the sum of its parts and we denote it  by $|\lambda|$. We denote by $\mathcal{P}[n]$ the set of all partitions of length at most $n$. The \emph{Young diagram} of the partition $\lambda$ is the left and top justified collection of boxes such that row $i$ contains $\lambda_i$ boxes. We denote it again by $\lambda$. For Young diagrams $\lambda$ and $\mu$ such that $\lambda \supset \mu $ (i.e., $\lambda_i \geq \mu_i \text{ }\forall i $), the \emph{skew diagram} $\lambda / \mu$ is obtained by removing the boxes of $\mu$ from those of $\lambda$.
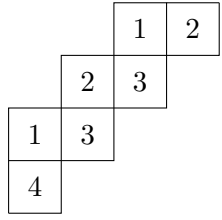
\begin{figure}
     \centering
     \begin{tikzpicture}
    \draw (0,0) -- (0,1.4) -- (2.1, 1.4) -- (2.1, 2.8) -- (2.8,2.8);
    \draw (2.8,2.8) -- (2.8,2.1) -- (0.7, 2.1) -- (0.7, 0) -- (0,0);
    \draw (0, 0.7) -- (1.4, 0.7) -- (1.4, 2.8) -- (2.1, 2.8);
    \draw (0.35, 0.35) node {$4$};
    \draw (0.35, 1.05) node {$1$};
    \draw (1.05, 1.05) node {$3$};
    \draw (1.05, 1.75) node {$2$};
    \draw (1.75, 1.75) node {$3$};
    \draw (1.75, 2.45) node {$1$};
    \draw (2.45, 2.45) node {$2$};
\end{tikzpicture}
     \caption{A skew tableau of shape $(4,3,2,1)/(2,1)$, weight $(2,2,2,1)$ and reverse-row reading word 2132314.}
     \label{fig:my_label}
\end{figure}

Given partitions $\lambda \supset \mu$, a \emph{semi-standard skew tableau} of shape $\lambda / \mu$ is a filling of the skew diagram $\lambda / \mu$ that is weakly increasing along the rows (from left to right) and strictly increasing along the columns (from top to bottom). A semi-standard tableau of shape $\lambda $ is just a semi-standard skew tableau of shape $\lambda/\text{empty} $ . We denote by $\Tab(\lambda / \mu)$ the set of all semi-standard skew tableaux of shape $\lambda / \mu$ and $\Tab(\lambda / \mu, n) $ the subset of $\Tab(\lambda / \mu)$ where the fillings are all $\leq n$. It will be convenient to let $\Tab(\lambda / \mu)$ be the empty set if $\lambda, \mu$ are partitions with $\lambda \not\supset \mu$. The \emph{weight} of a tableau $T \in \Tab(\lambda /\mu, n)$ is defined as $wt(T) = (t_1, t_2, \cdots , t_n)$, where $t_i$ is the number of times $i$ occurs in $T$. A standard $(\text{skew})$ tableau $ T$ is a semi-standard $(\text{skew})$ tableau of the same shape in which $1,2,\cdots, k$ appears exactly once, where $k$ is the number of boxes in $ T$. For $T\in \Tab(\lambda / \mu)$, we write $b_T$ to denote the \emph{reverse-row reading word of $T$} which is the word obtained by reading $T$ right to left and from top to bottom.

Fix $n$ a positive integer. A flag $\Phi$ = $(\Phi _1, \Phi _2, \cdots , \Phi_n)$ is a weakly increasing sequence of positive integers such that\footnote{In the literature, a flag need not have $\Phi_n = n$, but for our purposes it is sufficient to consider only such flags.} $\Phi_n = n$. For $\lambda, \mu \in \mathcal{P}[n]$ and flag $\Phi$ = $(\Phi _1, \Phi _2,\cdots, \Phi_n)$, $\Tab(\lambda / \mu, \Phi)$ is the set of all elements $T$ in $\Tab(\lambda /\mu)$ such that the entries in row $i$ of $T$ are at most $\Phi _i$ for $1 \leq i \leq n$. Following Reiner-Shimozono \cite{RS}, we define the {\em flagged skew Schur polynomial} $$s_{\lambda / \mu}(X_{\Phi}) = \sum _{T} \characx^{wt(T)}$$ where $T$ varies over $\Tab(\lambda / \mu, \Phi)$ and for $t = (t_1, t_2, \cdots, t_n) \in \mathbb{Z}^n _{+}$, $\characx^t$ denotes the monomial $x_1 ^{t_1} x_2 ^{t_2} \cdots x_n ^{t_n}$.
When $\Phi =(n, n, \cdots, n)$, these reduce to the skew Schur polynomials $s_{\lambda / \mu}(x_1, x_2, \cdots, x_n)$. When $\mu$ is the empty partition, they become the flagged Schur polynomials $s_{\lambda}(X_{\Phi})$, which coincide with key polynomials corresponding to $312$-avoiding permutations \cite[theorem 14.1]{PS}.
The flagged skew Schur polynomials $s_{\lambda / \mu}(X_\Phi)$ also have a representation theoretic interpretation as characters of certain Borel modules called {\em flagged Schur modules} \cite{rs-percentage}.	

For $1 \leq i \leq n-1$, define the Demazure operator $T_i$ on the ring of polynomials in the variables $x_1, x_2 \cdots, x_n$ as follows:
$$(T_i f)(x_1, x_2 \cdots, x_n) = \frac{x_i\, f(x_1, x_2,  \cdots, x_n) - x_{i+1}\, f(x_1, \cdots, x_{i-1}, x_{i+1}, x_i, x_{i+2},  \cdots, x_n )}{x_i - x_{i+1}}$$
For $w \in \mathfrak{S}_n$ (the symmetric group), we define $$T_w =  T_{i_1}T_{i_2} \cdots T_{i_k}$$ where $s_{i_1}s_{i_2} \cdots s_{i_k}$ is a reduced expression for $w$. This is well-defined because the $T_i$'s satisfy the braid relations.

For $\alpha \in \mathbb{Z}^n_{+}$, let $\alpha ^{\dagger}$ be the partition obtained by sorting the parts of $\alpha$ in descending order and let $\omega \in \mathfrak{S}_n$ be any permutation such that $\omega \alpha ^{\dagger} = \alpha$ (here, the action of $\omega$ is the usual left action of $\mathfrak{S}_n$ on $n$-tuples). We recall that the {\em key polynomial} $\key_\alpha$ is defined to be $\key_{\alpha} = T_{\omega}(\characx^{\alpha ^{\dagger}})$, and that this is independent of the choice of $\omega$. A polynomial $f$ is said to be {\em key-positive} if it is a sum of key polynomials. If $f$ is key-positive, then  $\characx^\lambda\, f$ and  $T_w(f)$ are also key-positive, for all $\lambda \in \mathcal{P}[n]$ and all $w \in \mathfrak{S_n}$. The former follows from a theorem of Joseph \cite{Joseph1} and the latter from the fact that a
composition of Demazure operators is itself a Demazure operator.
Reiner-Shimozono \cite[Theorem 20]{RS} showed that the flagged skew Schur polynomial $ s_{\lambda / \mu}(X_{\Phi})$ is key-positive. 

Now, if $w_0$ denotes the longest element of $\mathfrak{S}_n$,  we have 
\begin{equation}
  \label{key symmetrization}
  T_{w_0}(\key_{\alpha}) = s_{\alpha ^{\dagger}}(x_1, x_2, \cdots, x_n), 
\end{equation}
the Schur polynomial indexed by ${\alpha ^{\dagger}}$. More generally, (since the key polynomials form a $\mathbb Z$-basis of the polynomial ring in $n$ variables \cite[corollary 7]{RS}) given any polynomial $f = f(x_1, x_2, \cdots, x_n)$, we have that $T_{w_0}(f)$ is a symmetric polynomial, which can therefore be expanded in the basis of Schur polynomials. If further $f$ is key-positive, then equation~\eqref{key symmetrization} shows that $T_{w_0}(f)$ is Schur-positive, i.e., a sum of Schur polynomials.
This leads us to the main objects of our study.
\begin{definition}\label{def:flagskewLR}
  For $\lambda, \mu, \gamma \in \mathcal{P}[n]$ and flag $\Phi=(\Phi_1,\cdots ,\Phi_n)$, let 
  \begin{equation} \label{eq:defeqn}
    T_{w_0}(\characx^{\lambda} \,s_{\mu / \gamma}(X_{\Phi})) = \sum _{\nu \in \mathcal{P}[n]} c_{\lambda, \, \mu/\gamma} ^{\,\nu} (\Phi)  \hspace{0.2cm} s_{\nu}(x_1, x_2, \cdots, x_n)
    \end{equation}
  We call the coefficients $c_{\lambda, \, \mu/\gamma} ^{\,\nu} (\Phi)$ as the \textbf{flagged skew Littlewood-Richardson coefficients}.
  \end{definition}

By the remarks preceding equation~\eqref{key symmetrization}, it follows that the LHS of \eqref{eq:defeqn} is Schur positive, and thus the flagged skew Littlewood-Richardson coefficients are non-negative integers.
It is clear by definition that these coefficients are zero if $\mu \not \supset \gamma$. It will follow from Theorem~\ref{first result} below that they are also zero if $\nu \not\supset \lambda$.


These coefficients subsume many other extensions of the Littlewood-Richardson coefficients.  When $\Phi = (n, n, \cdots, n)$, these become Zelevinsky's extension \cite{Z} of the Littlewood-Richardson coefficients $c_{\lambda, \, \mu/\gamma} ^{\, \nu}$ defined by $s_{\lambda}(\textbf{x})\, s_{\mu / \gamma}(\textbf{x})= \sum_{\nu \in \mathcal{P}[n] } c_{\lambda,   \, \mu/\gamma} ^{\, \nu} \,s_{\nu}(\textbf{x})$. These in turn reduce to the usual Littlewood-Richardson coefficients when we further take $\gamma = (0, 0, \cdots, 0)$.

On the other hand, if we take $\gamma = (0, 0, \cdots, 0)$ but let $\Phi$ remain arbitrary, we get the \emph{w-refined Littlewood-Richardson coefficients} of \cite{krvfpsac} for \emph{312-avoiding permutations} $w$.

If we set $\lambda = (0, \dots, 0)$, we have $c_{\lambda, \,\mu/\gamma} ^{\,\nu} (\Phi) = \sum _{\, \alpha} c^{\,\mu / \gamma, \Phi}_{\,\alpha}$ where the sum runs over all compositions $\alpha$ that are obtained by permuting the parts of $\nu$. The coefficients on the right are the ones which appear in the flagged Littlewood-Richardson expansion of \cite[section 7]{RS}.

Our first result provides two combinatorial models for flagged skew LR coefficients (see \S \ref{sec:crys}, \ref{sec:hive} for undefined terms) that generalize those for LR coefficients:
\begin{theorem}
Let $\Phi=(\Phi_1, \Phi_2, \cdots, \Phi_n)$ be a flag and $\lambda,\mu,\nu,\gamma \in \mathcal{P}[n] $. Then,
\begin{enumerate}
    \item $c_{\lambda, \,\mu/\gamma} ^{\,\nu} (\Phi)$ is the cardinality of the set of all $\lambda$-dominant tableaux in $ \Tab (\mu /\gamma, \Phi )$ of weight $\nu - \lambda$. 
    \item $c_{\lambda, \,\mu/\gamma} ^{\,\nu} (\Phi)$ is the number of the integral points of the flagged skew hive polytope $ \SHive (\lambda, \mu, \gamma, \nu, \Phi)$.
\end{enumerate}
\label{first result}
\end{theorem}
Our proof of Theorem~\ref{first result} hinges on understanding the crystal structure on the set of flagged skew tableaux $\Tab(\mu / \gamma, \Phi)$. We show in particular that this set is a disjoint union of Demazure crystals; this lifts the key-positivity result of Reiner-Shimozono \cite[Theorem 20]{RS} from the level of characters to that of crystals.

The main theorem of this paper is the following {\em saturation property} of the flagged skew LR coefficients:
\begin{theorem} \label{main theorem}
    Let $\Phi$ be a flag and $\lambda,\mu,\nu,\gamma \in \mathcal{P}[n] $. Then,
    $$ c_{k\lambda, \,k\mu/k\gamma} ^{\,k\nu} (\Phi) > 0 \text{ for some } k \geq 1 \implies c_{\lambda, \,\mu/\gamma} ^{\,\nu} (\Phi) > 0 $$
\end{theorem}
We remark that, as in the classical LR case, the stronger converse statement holds. Scaling $\lambda, \mu, \gamma, \nu$ by $k$ also dilates the polytope $\SHive (\lambda, \mu, \gamma, \nu, \Phi)$ by the factor $k$. Thus, $c_{\lambda, \,\mu/\gamma} ^{\,\nu} (\Phi) > 0$ implies that $ c_{k\lambda, \,k\mu/k\gamma} ^{\,k\nu} (\Phi) > 0$ for every $k \geq 1$, by Theorem~\ref{first result}.

\medskip

To prove Theorem~\ref{main theorem}, we construct an affine linear isomorphism between $\SHive (\lambda, \mu, \gamma, \nu) $ and a certain hive polytope, which preserves flags and integral points. Then we deduce theorem \ref{main theorem} from the corresponding theorem for hive polytopes \cite[Theorem 1.4]{krvfpsac}. 

\section{The crystal $\Tab(\mu/ \gamma, \Phi)$}\label{sec:crys}
The purpose of this section is to prove that the subset $\Tab(\mu / \gamma, \Phi)$ of the type $A_{n-1}$ crystal\\
$\Tab(\mu / \gamma, n)$ is a disjoint union of Demazure crystals. This is the key step in proving the first part of theorem \ref{first result} which we will see at the end of this section.  
 
By a crystal of type $A_{n-1}$ (see section 2.2 of \cite{bump-sch}), we mean a finite and non-empty set $\mathcal{B}$ together with maps $$e_i, f_i : \mathcal{B} \rightarrow \mathcal{B} \sqcup \{0\}$$

$$wt: \mathcal{B} \rightarrow \mathbb{Z}^{n}$$
where $i \in \{1,2,\cdots, n-1\}$ and $0 \not\in \mathcal{B}$ is an auxiliary element, satisfying the following conditions:
\begin{enumerate} 
    \item If $x, y \in \mathcal{B}$ then $e_i(x) = y$ if and only if $f_i(y) = x$. In this case it is assumed that $$wt(y) = wt(x) + (\epsilon_i - \epsilon_{i+1})$$ where $\epsilon_1, \epsilon_2, \cdots, \epsilon_{n}$ are the standard orthonormal vectors in $\mathbb{R}^{n}$.
    \item For all $x \in \mathcal{B}$ and $i \in \{1, 2, \cdots , n-1\}$ we require that $$\phi_i(x) = wt(x) \cdot (\epsilon_i - \epsilon_{i+1}) + \varepsilon _i (x)$$ where $\varepsilon_i(x)$ (resp. $\phi_i (y)$) is the maximum number of times $e_i$ (resp. $f_i$) can be applied to $x$ (resp. $y$) without making it $0$.
\end{enumerate}
The maps $e_i$ and $f_i$ are called the \emph{raising} and \emph{lowering} operators respectively.

\medskip
\noindent \textbf{Example:} The \emph{Standard type $A_{n-1}$ crystal} is $\mathbb{B} = \{1, 2, \cdots, n\}$ where
\begin{equation*}
f_k(l) = \left\{
    \begin{array}{ll}
         k+1 & \quad \text{if} \quad l=k  \\
         0 & \quad \text{otherwise}
    \end{array}     \right.
\end{equation*} 
Also, $wt(i) = \epsilon _i \in \mathbb{Z}^{n}$. This crystal can be depicted by the following ``crystal graph'':

\begin{figure}[h]
\begin{center}
    \begin{tikzpicture}
        \draw (0,0) rectangle (0.5,0.5);
        \node at (0.25, 0.25) {1};
        \draw[->] (0.9,0.25) -- (1.6, 0.25);
        \node at (1.2, 0.5) {\small 1};
        \draw (2,0) rectangle (2.5,0.5);
        \node at (2.25, 0.25) {2};
        \draw[->] (2.9,0.25) -- (3.6, 0.25);
        \node at (3.2, 0.5) {\small 2};
        \draw[dashed] (4,0.25) -- (5,0.25);
        \draw[->] (5.4,0.25) -- (6.1, 0.25);
        \node at (5.7, 0.5) {\small $n-1$};
        \draw (6.5,0) rectangle (7,0.5);
        \node at (6.75, 0.25) {$n$};
    \end{tikzpicture}
\end{center}
\end{figure}

The {\em crystal graph\/} associated to a type $A_{n-1}$ crystal $\mathcal{C}$ is an edge-coloured (the colours being $1, 2, \cdots n-1$) directed graph whose vertex set is the underlying set of the crystal. An edge with colour $k$ originates from $x \in \mathcal{C}$ and terminates at $y \in \mathcal{C}$ if and only if $f_k (x) = y$. We say a crystal is {\em connected\/} if its crystal graph is connected (viewed as an undirected graph).

A subset $\mathcal{C}'$ of a crystal $\mathcal{C}$ which is a union of connected components of $\mathcal{C}$ inherits a crystal structure from that of $\mathcal{C}$. In this case, we call $\mathcal{C}'$ a {\em full-subcrystal\/} of $\mathcal{C}$.
\subsection{Tensor Product of Crystals} \label{ss:tpcrystal}
\null \quad If $\mathcal{B}$ and $\mathcal{C}$ are type $A_{n-1}$ crystals, there is a natural notion of the tensor product crystal $\mathcal{B} \otimes \mathcal{C}$. As a set it is $\{x \otimes y: x \in \mathcal{B}, y \in \mathcal{C}\}$ (where $x\otimes y$ is just a symbol). We define $wt(x \otimes y)$ to be $wt(x) + wt(y)$. The raising and lowering operators are defined as follows:
\begin{equation}\label{e action}
 e_i(x \otimes y) = \left\{ 
    \begin{array}{ll}
        e_i(x) \otimes y & \quad \text{if} \quad \varepsilon_i(y) \leq \phi_i(x)  \\
        x \otimes e_i(y) & \quad \text{if} \quad \varepsilon_i(y) > \phi_i(x) 
    \end{array}   \right. 
\end{equation}
and 
\begin{equation}\label{f action}
 f_i(x \otimes y) = \left\{ 
    \begin{array}{ll}
        f_i(x) \otimes y & \quad \text{if} \quad \varepsilon_i(y) < \phi_i(x)  \\
        x \otimes f_i(y) & \quad \text{if} \quad \varepsilon_i(y) \geq \phi_i(x) 
    \end{array}   \right. 
\end{equation}
It is understood that $x \otimes 0 = 0 \otimes y = 0$.

We define the character of a crystal $\mathcal{C}$ to be $ch(\mathcal{C}) := \sum _{u \in \mathcal{C}} {\characx}^{wt(u)}$. From the definition of the tensor product crystal, it is elementary to observe that the character of the tensor product is equal to the product of the characters.  The tensor product is associative\footnote{The convention for tensor products in \cite{bump-sch} differs from the widely-used convention (that is also employed in this paper).} - see \cite[\S 2.3 and remark 1.1]{bump-sch}.

The set $\mathbb{B}^{\otimes k}$ is usually called the crystal of words (of length $k$). An element $\zeta$ of  $\mathbb{B}^{\otimes k}$ is said to be a {\em dominant word} if $e_i \zeta = 0 $ for all $i$.

\begin{remark}
	Let $\lambda, \mu \in \mathcal{P}[n]$ such that $|\lambda| - |\mu| = k$. Then $\Tab(\lambda / \mu, n)$ is given the structure of a type $A_{n-1}$ crystal by the following embedding into  $\mathbb{B}^{\otimes k}$ (where $b_T$ is the reverse-row reading word of $T$, defined in \S\ref{s:prelims}):
    $$T \mapsto b_T = w_1 w_2 \cdots w_k \mapsto w_1 \otimes w_2 \otimes \cdots \otimes w_k$$
The image under this embedding is a full-subcrystal of~$\mathbb{B}^{\otimes k}$ (see ~\cite[Section 3.1]{bump-sch}).
\end{remark}
\subsection{Demazure Crystals}
Let $\lambda \in \mathcal{P}[n]$ and $\omega$ a permutation in the symmetric group $\mathfrak{S}_n$. For any reduced expression $s_{i_1} s_{i_2} \cdots s_{i_p}$ of $\omega$, the {\em Demazure crystal\/} $\mathcal{B}_{\omega} (\lambda)$ is defined by:
\begin{equation} \label{eq:demcrys}
  \mathcal{B}_{\omega} (\lambda) := \{f_{i_1} ^{k_1} f_{i_2} ^{k_2} \cdots f_{i_p} ^{k_r}  T_{\lambda} ^0 : k_j \geq0 \} \textcolor{black}{\setminus \{0\}} \subset \Tab(\lambda, n)
\end{equation}
where $T_{\lambda} ^0$ is the unique {\em dominant} tableau of shape~$\lambda$ (i.e., with shape and weight both equal to $\lambda$).
\begin{remark}
\label{dom-in-dem}
For any $\omega \in \mathfrak{S}_n$ and $\lambda \in \mathcal{P}[n]$, $T_{\lambda} ^0$ is the unique element in $\mathcal{B}_{\omega}(\lambda)$ such that $e_i(T_{\lambda} ^0) = 0$ for all $i$. 
In \eqref{eq:demcrys} above, we could replace $T_{\lambda} ^0$ with any other dominant word $b_{\lambda} \in \mathbb{B}^{\otimes |\lambda|}$ of weight $\lambda$. We thereby
obtain a subset of $\mathbb{B}^{\otimes |\lambda|}$:
$$\mathcal{B}_{\omega} (b_\lambda) := \{f_{i_1} ^{k_1} f_{i_2} ^{k_2} \cdots f_{i_p} ^{k_r}  b_\lambda : k_j \geq0 \} \textcolor{black}{\setminus \{0\}}$$
which is isomorphic to $\mathcal{B}_{\omega} (\lambda)$ as crystals, i.e., there is a weight-preserving bijection between these sets which intertwines the crystal raising and lowering operators (where defined). We also refer to $\mathcal{B}_{\omega} (b_\lambda)$ as a Demazure crystal in what follows, and write (by abuse of notation) $\mathcal{B}_{\omega} (\lambda)$ in place of $\mathcal{B}_{\omega} (b_\lambda)$.

\end{remark}
The following proposition is the refined Demazure character formula in \cite{kash}.
\begin{proposition}    
Let $\lambda \in \mathcal{P}[n]$ and $\omega \in \mathfrak{S}_n$. Then, $ch(\mathcal{B}_{\omega}(\lambda)) = \key_{\omega\lambda}$.
\end{proposition}

\noindent \textbf{Examples:}
\begin{itemize}
    \item $\mathcal{B}_{\omega_0}(\lambda) = \Tab(\lambda, n)$. This crystal is denoted as $B(\lambda)$ in the literature.
    \item $\mathcal{B}_{s_{k-1} \cdots s_2s_1}((1))$ = $\Tab((1), k) \subset \Tab((1), n)$ for $1 \leq k \leq n$. We will denote this Demazure crystal by $\mathbb{B}_k$.
\end{itemize}
A subset $\mathcal{S}$ of a crystal $\mathcal{C}$ is said to have the \emph{string property} if $\forall i \in \{1, 2, \cdots, n-1\}$ and $\forall x \in \mathcal{S}$ such that $e_i (x) \neq 0$, we have
\begin{enumerate}
    \item $e_i (x) \in \mathcal{S},$
    \item $f_i (x) \neq 0$ implies $f_i (x) \in \mathcal{S}$
\end{enumerate}
\begin{proposition}
\label{a}
~\cite[Proposition 3.3.5]{kash}
    For any $\lambda \in \mathcal{P}[n]$ and $\omega \in \mathfrak{S}_n$, the Demazure subcrystal $\mathcal{B}_{\omega} (\lambda)$ of $\Tab(\lambda, n)$ has the string property.
\end{proposition} 
\begin{remark}
    The converse of proposition \ref{a} is not true in general (see ~\cite{sami}).  
\end{remark}
A characterization of when a tensor product of Demazure crystals decomposes into Demazure crystals was given in \cite{kuono}. Following \cite{kuono}, a different characterization was obtained in \cite{sami} as follows:
\begin{theorem}
    
    \label{sami-thrm}
    \cite[Theorem 1.2]{sami} For $\lambda, \mu \in \mathcal{P}[n]$ and $\omega, \tau \in \mathfrak{S}_n$, the subset $\mathcal{B}_{\omega} (\lambda) \otimes \mathcal{B}_{\tau} (\mu)$ of $\Tab(\lambda, n) \otimes \Tab(\mu, n)$ is a disjoint union of Demazure crystals if and only if $\mathcal{B}_{\omega} (\lambda) \otimes \mathcal{B}_{\tau} (\mu)$ has the string property.
\end{theorem}
    Subsets of crystals that exhibit the string property are referred to as {\em extremal} in \cite{sami}.
The following proposition is a strengthening of \cite[Proposition 8.1]{sami}. 
\begin{proposition}
     \label{lemma for sami}Let $X_1$ and $X_2$ be subsets of crystals $C_1$ and $C_2$ respectively. Assume that $\forall i \in \{1,2, \cdots, n-1\}$ $\exists x_i^{\dagger} \in X_2$ such that $e_i(x_i^{\dagger}) = 0$ (i.e., for each $i$, $X_2$ contains a head of some $i$-string). If $X_1 \otimes X_2$ has the string property, then $X_1$ has the string property.
\end{proposition}
\begin{proof}
    Let $i \in \{1,2, \cdots, n-1\}$ be arbitrary. Suppose $u \in X_1$ such that $e_i(u) \neq 0$. 
    Then $e_i (u \otimes x_i^{\dagger}) = e_i(u) \otimes x_i^{\dagger} \neq 0 \text{ since } \varepsilon_i(x_i^{\dagger}) \leq \phi_i(u)$. Therefore by the string property of $X_1 \otimes X_2$ we have  $e_i(u) \otimes x_i^{\dagger} \in X_1 \otimes X_2$, which implies that $e_i(u) \in X_1$.
    
    Additionally, if $f_i (u) \neq 0$ then $f_i (u\otimes x_i^{\dagger}) = f_i(u) \otimes x_i^{\dagger} \neq 0$ since $ \varepsilon_i(x_i^{\dagger}) < \phi_i(u) $. Therefore by the string property of $X_1 \otimes X_2$ we have $f_i(u)\otimes x_i^{\dagger} \in X_1 \otimes X_2$, which implies that $f_i(u) \in X_1$. 
\end{proof}
\begin{corollary}
\label{cor-for-sami}
    Let $D_1, D_2, \cdots, D_k$ be Demazure crystals. If $D_1 \otimes D_2 \otimes \cdots \otimes D_k$ has the string property, then it is a disjoint union of Demazure crystals.
\end{corollary}
\begin{proof}
    We prove by induction on $k$. The case $k =1$ is straight forward. Suppose $k > 1$. Then it follows from proposition \ref{lemma for sami} that $D_1 \otimes D_2 \otimes \cdots \otimes D_{k-1}$ has string property since $D_k$ has $x_i ^{\dagger}$ such that $ e_i(x_i^{\dagger})=0 $ $ \forall i $. Therefore by induction hypothesis $D_1 \otimes D_2 \otimes \cdots \otimes D_{k-1}$ is a disjoint union of Demazure crystals (say = $\bigsqcup \Tilde{D}$). Now, $$D_1 \otimes D_2 \otimes \cdots \otimes D_k = \bigsqcup (\Tilde{D} \otimes D_k)$$ 
    Observe that if $i \in \{1, 2, \cdots, n-1\}$ and $x\otimes y \in \Tilde{D} \otimes D_k$ then $e_i(x \otimes y) \in (\Tilde{D} \otimes D_k) \cup \{0\}$ because of proposition \ref{a}.

    Let $x \otimes y \in \Tilde{D} \otimes D_k$ such that $e_i (x \otimes y) \neq 0$ and $f_i(x \otimes y) \neq 0.$ It follows that $f_i(x \otimes y) \in \Tilde{D}\otimes D_k$ because $e_i(f_i(x \otimes y)) = x \otimes y$ and the fact that the decomposition $D_1 \otimes D_2 \otimes \cdots \otimes D_{k-1} = \bigsqcup \Tilde{D}$ is disconnected. Therefore $\Tilde{D} \otimes D_k$ has string property. By theorem \ref{sami-thrm}, it now follows that $\Tilde{D} \otimes D_k$ is a disjoint union of Demazure crystals. 
\end{proof}
Let $\Phi$ be a flag and $\rho \in \mathbb{Z}_+ ^n$ be a composition of $k \in \mathbb Z _+$ (i.e., $\rho_1 + \rho_2 + \cdots + \rho_n = k$). Define the subset of $\mathbb{B}^{\otimes k}$: $$\mathbb{B}_{\Phi}  ^{\rho}:= \mathbb{B}_{\Phi_1} ^{\otimes \rho_1} \otimes \mathbb{B}_{\Phi_2} ^{\otimes \rho_2} \otimes \cdots \otimes \mathbb{B}_{\Phi_n}^{\otimes \rho_n}$$
\begin{lemma}\label{b rho lemma}    $\mathbb{B}_{\Phi} ^{\rho}$ is a disjoint union of Demazure crystals.
\end{lemma}
\begin{proof}
    By corollary \ref{cor-for-sami}, it is sufficient to show that $\mathbb{B}_{\Phi}^{\rho}$ has string property.

    Let $u = u_1 \otimes u_2 \otimes \cdots \otimes u_k \in \mathbb{B}_{\Phi}^{\rho}$. Suppose that $e_i(u) \neq 0$. Then $e_i(u) \in \mathbb{B}_{\Phi}^{\rho}$, because of proposition \ref{a} and the fact that $\mathbb{B}_{\Phi_i}$'s are Demazure crystals.

    Now suppose furthermore that $f_i(u) \neq 0$. Let $t$ be the index where $e_i$ acts in $u$. i.e., 
    \begin{align*}
      u = u_1 \otimes \cdots \otimes u_{t-1} \otimes (i+1) \otimes u_{t+1} \otimes  \cdots \otimes u_k\\
      e_i(u) = u_1 \otimes \cdots \otimes u_{t-1} \otimes i \otimes u_{t+1} \otimes \cdots \otimes u_k
    \end{align*}

    Define $u' = u_1 \otimes \cdots \otimes u_{t-1}$ and $u'' = (i+1) \otimes u_{t+1} \otimes \cdots \otimes u_k$. Then by (\ref{e action}) it follows that $\varepsilon_i(u'') > \phi_i(u')$. Therefore, (\ref{f action}) implies $f_i (u) = u' \otimes f_i(u'')$. The action of $f_i$ on $u''$ amounts to changing a $u_{t_0}$ ($t_0>t$) which is an $i$ to an $i+1$. But since $\Phi_t \geq i+1$ (by definition of $u$) and $\Phi$ is weakly increasing, it follows that $f_i(u) \in \mathbb{B}_{\Phi}^{\rho}$.
\end{proof}
\begin{theorem}
\label{crystal-thrm}
    Let $\Phi$ be a flag and $\mu, \gamma \in \mathcal{P}[n]$ such that $\mu \subset \gamma$. Then $\Tab(\mu / \gamma, \Phi)$ is a disjoint union of Demazure crystals.
\end{theorem}
\begin{proof}
Define the composition $\rho$ by $\rho_i = \mu_i - \gamma_i, \hspace{0.2cm} 1 \leq i \leq n$. Let $k = |\rho| = \rho_1 + \cdots + \rho_n$. Since $\Tab(\mu / \gamma, n)$ is a full subcrystal of $\mathbb{B}^{\otimes k}$, the theorem follows from lemma \ref{b rho lemma} because $$\Tab(\mu / \gamma, \Phi) = \Tab(\mu / \gamma, n) \cap \mathbb{B}_{\Phi}^{\rho}$$
\end{proof}
\begin{remark}
\cite{RS} proves the above theorem at the character level (i.e., the key positivity of $s_{\mu / \gamma}(\Phi)$). Theorem \ref{crystal-thrm} is however implicit in \cite{RS} and allows us to compute the explicit decomposition of $\Tab(\lambda / \mu, \Phi)$ into Demazure crystals. A short sketch is deferred to the appendix. 
\end{remark}
\begin{remark}
    \textcolor{black}{It is important that we assume $\Phi$ to be weakly increasing. For example, if $\mu = (3,2)$, $\gamma = (1,0)$ and $\Phi = (3,2)$, then $\Tab(\mu /\gamma, \Phi)$ does not even have the string property.}
\end{remark}
Consider the tensor product  $\mathcal{C}(\lambda; \mu, \gamma, \Phi) = \Tab(\lambda, \Phi_0) \otimes \Tab(\mu / \gamma, \Phi)$. Here $\Phi_0 = (1,2,3, \cdots)$ is the standard flag. The Demazure crystal $\Tab(\lambda, \Phi_0)$ is the singleton set containing $T_{\lambda} ^0$.

\begin{proposition} \label{extension of main theorem}
    $\mathcal{C}(\lambda; \mu, \gamma, \Phi)$ is a disjoint union of Demazure crystals.
\end{proposition}
\begin{proof}
By Theorem~\ref{crystal-thrm}, we have $\Tab(\mu / \gamma, \Phi) = \bigsqcup_p D_p$ where each $D_p$ is a Demazure crystal. Then $$\mathcal{C}(\lambda; \mu, \gamma, \Phi) = \Tab(\lambda, \Phi_0) \otimes (\bigsqcup D_p) = \bigsqcup \Tab(\lambda, \Phi_0) \otimes D_p $$ 
But $\Tab(\lambda, \Phi_0) \otimes D_p$ is a disjoint union of Demazure crystals by Joseph's theorem \cite{Joseph1}. This fact also follows from theorem \ref{sami-thrm} because $\Tab(\lambda, \Phi_0) \otimes D$ has the string property as we show below:

For $u \otimes v \in \Tab(\lambda, \Phi_0) \otimes D$, $e_i(u \otimes v) \neq 0$ only if $e_i(u \otimes v) = u \otimes e_i(v)$. This implies that $\varepsilon_i(v) > \phi_i(u)$. By the tensor product rule we therefore have $f_i(u \otimes v) = u \otimes f_i(v)$. By assumptions, $e_i(v) \neq 0$ and $f_i(v) \neq 0$. Since $D$ is a Demazure crystal, by proposition \ref{a} it follows that $f_i(v) \in D$ and hence $f_i(u \otimes v) \in \Tab(\lambda, \Phi_0) \otimes D$.
\end{proof}
\begin{corollary} (\cite[Theorem 20]{RS})
    $\characx^{\lambda}\,s_{\mu / \gamma} (\Phi) = ch(\mathcal{C}(\lambda; \mu, \gamma, \Phi))$ is key-positive. 
\end{corollary}

A skew tableau $T \in Tab(\mu/\gamma, n)$ is {\em $\lambda$-dominant} if  the concatenated word $b_{T_{\lambda}^0} *b_T$ is a dominant word. We now prove  the first part of theorem \ref{first result}.
\begin{theorem}\label{thm:firstpart}
$c_{\lambda, \, \mu /\gamma} ^{\,\nu} (\Phi)$  is the cardinality of the set $\Tab_{\lambda} ^{\nu} (\mu /\gamma, \Phi)$ of all $\lambda$-dominant tableaux in $\Tab(\mu / \gamma, \Phi)$ of weight $\nu - \lambda$.
\end{theorem}
\begin{proof}
  By Proposition \ref{extension of main theorem} (and Remark~\ref{dom-in-dem}), for all $\nu \in \mathcal{P}[n]$ there exists a multi-subset $\mathcal{W}(\nu) \subseteq \mathfrak{S}_n$ such that
  \begin{equation} \label{eq:clmnp}
    \mathcal{C}(\lambda; \mu, \gamma, \Phi) = \bigsqcup _{\nu \in \mathcal{P}[n]} \bigsqcup _{w \in \mathcal{W}(\nu)} \mathcal{B}_w (\nu)
    \end{equation}
    Taking characters, we obtain
    \[\characx^{\lambda}\,s_{\mu / \gamma} (\Phi) = \sum_{\nu}\sum_{w \in \mathcal{W}(\nu)} \key_{w\nu}\]
    Applying $T_{w_0}$ and using proposition \ref{key symmetrization}  gives $T_{w_0}(\characx^{\lambda}\,s_{\mu / \gamma} (\Phi)) = \sum_{\nu} |\mathcal{W}(\nu)| s_{\nu}(\characx)$. Thus $$c^{\,\nu}_{\lambda, \,\mu /\gamma} (\Phi) = |\mathcal{W}(\nu)|$$
    Now, by definition of $\lambda$-dominance, the number of elements $\zeta \in \mathcal{C}(\lambda; \mu, \gamma, \Phi)$ of weight $\nu$ satisfying $e_i \zeta =0 $ for all $i$ is precisely $|Tab^{\nu} _{\lambda}(\mu /\gamma, \Phi)|$.
    On the other hand, in the RHS of equation~\eqref{eq:clmnp}, each $\mathcal{B}_w (\nu)$ has a unique element $\xi$ such that $e_i \xi =0 $ for all $i$; this element has weight $\nu$. Thus, the number of elements $\zeta$ as above  is also equal to  $|\mathcal{W}(\nu)|$. Putting all these together establishes Theorem~\ref{thm:firstpart}. 
\end{proof}

\section{A hive model for  flagged skew Littlewood-Richardson coefficients}\label{sec:hive}
In this section we define the skew hive polytope and its faces corresponding to flags $\Phi$. We then prove the second part of theorem \ref{first result}.
\subsection{Skew GT patterns}\label{sec:skewgt}
Given $m,n \geq 1 $, a skew Gelfand-Tsetlin pattern is an array of real numbers $\{x_{ij}: 0 \leq i \leq m, \hspace{0.2cm} 1 \leq j \leq n\}$ satisfying the following inequalities:
\begin{center}
$NE_{ij}= x_{ij} - x_{(i-1)j} \geq 0 \quad \quad \quad 1 \leq i \leq m; \quad 1 \leq j \leq n$\\ 
$SE_{ij}=  x_{(i-1)j} - x_{i (j+1)} \geq 0 \quad \quad \quad 1 \leq i \leq m; \quad 1\leq j \leq n-1 $
\end{center}
The above inequalities simply mean that the consecutive rows interlace. Hence we arrange the rows in the shape of a parallelogram as follows (shown for $m=n=4$):

\medskip
\begin{center}
	\begin{tikzpicture}
		\draw (2,2*1.732) node {$x_{01}$};
		\draw (3,2*1.732) node {$x_{02}$};
		\draw (4,2*1.732) node {$x_{03}$};
		\draw (5,2*1.732) node {$x_{04}$};
		\draw (1.5,1.5*1.732) node {$x_{11}$};
		\draw (2.5,1.5*1.732) node {$x_{12}$};
		\draw (3.5,1.5*1.732) node {$x_{13}$};
		\draw (4.5,1.5*1.732) node {$x_{14}$};
		\draw (1,1.732) node {$x_{21}$};
		\draw (2,1.732) node {$x_{22}$};
		\draw (3,1.732) node {$x_{23}$};
		\draw (4,1.732) node {$x_{24}$};
		\draw (0.5,0.5*1.732) node {$x_{31}$};
		\draw (1.5,0.5*1.732) node {$x_{32}$};
		\draw (2.5,0.5*1.732) node {$x_{33}$};
		\draw (3.5,0.5*1.732) node {$x_{34}$};
		\draw (0,0) node {$x_{41}$};
		\draw (1,0) node {$x_{42}$};
		\draw (2,0) node {$x_{43}$};
		\draw (3,0) node {$x_{44}$};
	\end{tikzpicture}
\end{center}

For $\mu, \gamma \in \mathcal{P}[n]$, such that $\gamma \subset \mu$, the skew Gelfand-Tsetlin polytope
$\GT(\mu/\gamma, m)$ is the set of all skew Gelfand-Tsetlin patterns $(x_{ij})$ with $0 \leq i \leq m,\, 1 \leq j \leq n$ satisfying $x_{0j}=\gamma_j$, $x_{mj}=\mu_j$ for all $j$. Define, $$\GT_{\mathbb{Z}}(\mu/\gamma, m) := \{X = (x_{ij}) \in \GT(\mu / \gamma, m): x_{ij}\in \mathbb{Z}\}$$
 
 In the sequel, we will only have occasion to consider the case when $m=n$.
Consider the map $$\Upsilon: \GT_{\mathbb{Z}}(\mu /\gamma, n) \rightarrow  \Tab(\mu / \gamma,n)$$ where $X = (x_{ij}) \in \GT_{\mathbb{Z}}(\mu /\gamma, n)$ maps to the unique tableau $\Upsilon(X)$ in $\Tab(\mu / \gamma)$ such that the number of $i$ that appears in the $j^{th}$ row of $\Upsilon(X)$ is $x_{ij} - x_{(i-1)j}$.
\begin{figure}[h]
\begin{center}
\begin{tikzpicture}[scale=0.8]
	\draw (2,2*1.732) node {$2$};
	\draw (3,2*1.732) node {$1$};
	\draw (4,2*1.732) node {$0$};
	\draw (5,2*1.732) node {$0$};
	\draw (1.5,1.5*1.732) node {$3$};
	\draw (2.5,1.5*1.732) node {$2$};
	\draw (3.5,1.5*1.732) node {$0$};
	\draw (4.5,1.5*1.732) node {$0$};
	\draw (1,1.732) node {$4$};
	\draw (2,1.732) node {$3$};
	\draw (3,1.732) node {$0$};
	\draw (4,1.732) node {$0$};
	\draw (0.5,0.5*1.732) node {$4$};
	\draw (1.5,0.5*1.732) node {$3$};
	\draw (2.5,0.5*1.732) node {$2$};
	\draw (3.5,0.5*1.732) node {$0$};
	\draw (0,0) node {$4$};
	\draw (1,0) node {$3$};
	\draw (2,0) node {$2$};
	\draw (3,0) node {$1$};
\end{tikzpicture}
\begin{tikzpicture}
    \draw (0,0) node {\null};
    \draw[|->] (0.5,1.5) -- (1.5,1.5);
    \draw (1,1.75) node {$\Upsilon$};
    \draw (2.5,1.5) node {\null};
\end{tikzpicture}
\begin{tikzpicture}
    \draw (0,0) -- (0,1.4) -- (2.1, 1.4) -- (2.1, 2.8) -- (2.8,2.8);
    \draw (2.8,2.8) -- (2.8,2.1) -- (0.7, 2.1) -- (0.7, 0) -- (0,0);
    \draw (0, 0.7) -- (1.4, 0.7) -- (1.4, 2.8) -- (2.1, 2.8);
    \draw (0.35, 0.35) node {$4$};
    \draw (0.35, 1.05) node {$3$};
    \draw (1.05, 1.05) node {$3$};
    \draw (1.05, 1.75) node {$1$};
    \draw (1.75, 1.75) node {$2$};
    \draw (1.75, 2.45) node {$1$};
    \draw (2.45, 2.45) node {$2$};
\end{tikzpicture}
\end{center}
\caption{The skew GT pattern on the left maps to the skew tableau on the right under the map $ \Upsilon$.}\label{Tab to GT}
\end{figure}
The following statement is well-known - see for instance \cite[\S 3]{louck} (whose pattern drawing convention differs from ours by a vertical flip).
\begin{lemma} \label{louck_lemma}
 The map $\Upsilon$ is a bijection.
\end{lemma}

\subsection{Flagged skew GT patterns}\label{sec:flagskewGT}
We keep the notation of the previous subsection, but in addition assume that we are given a flag $\Phi =(\Phi_1,\cdots ,\Phi_n)$. Define the set of flagged skew GT patterns:
\begin{equation}\label{eq:flaggt}
  \GT(\mu / \gamma, \Phi) = \{(x_{ij})\in \GT(\mu / \gamma, n): x_{n,j} = x_{n-1,j}  = \cdots = x_{\Phi_j,j} \quad \forall 1 \leq j \leq n \}
  \end{equation}
and let $\GT_{\mathbb{Z}}(\mu / \gamma, \Phi)$ denote the set of integer points in this polytope. We have:
\begin{lemma}\label{lem:flags}
The map $\Upsilon$ restricts to a bijection between $\GT_{\mathbb{Z}}(\mu / \gamma , \Phi)$ and $\Tab(\mu / \gamma, \Phi)$.
\end{lemma}
\begin{proof}
    Let $ X =(x_{ij}) \in \GT_{\mathbb{Z}}(\mu / \gamma , \Phi)$. If $\Phi=(n,n,\cdots,n)$ then it is easy to see $ \Upsilon(X) \in \Tab(\mu/\gamma,\Phi) $. Otherwise let $k$ be the maximum such that $ \Phi_{k} \neq n $. Then the number of $ i \,(\Phi_j<i \leq n) $ that appear in the $j^{th}$ row $ (1 \leq j \leq k)$ of $\Upsilon(X)$ is $x_{ij} - x_{(i-1)j}=0$. Since $\forall j > k$, $  \Phi_j = n$ so for those $j$ in $j^{th}$ row of $\Upsilon(X)$ all entries are $ \leq \Phi_j(=n)$. Thus $ \Upsilon(X) \in \Tab(\mu/\gamma,\Phi)$. Also, if $ T \in \Tab(\mu/\gamma,\Phi) $ then $(i,j)^{th}$ entry of $\Upsilon^{-1}(T) $ is the number of entries $\leq i$ that appear in the $j^{th}$ row of $T$. So $ \Upsilon^{-1}(T) \in \GT_{\mathbb{Z}}(\mu / \gamma , \Phi)$.  
\end{proof}
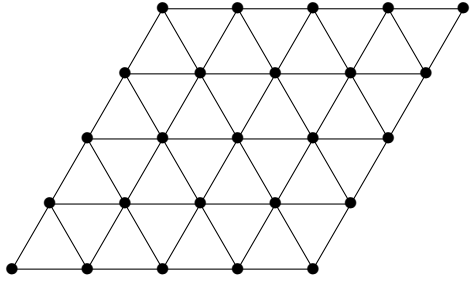
\begin{figure}[h]
    \centering
    \begin{center}
    \begin{tikzpicture}
        \node at (0,0) {$\bullet$};
        \node at (1,0) {$\bullet$};
        \node at (2,0) {$\bullet$};
        \node at (3,0) {$\bullet$};
		\node at (4,0) {$\bullet$};
        \draw (0,0) -- (1,0) -- (2,0) -- (3,0);
        \draw (3,0) -- (4,0);
        \node at (0.5,0.5*1.732) {$\bullet$};
        \node at (1.5,0.5*1.732) {$\bullet$};
        \node at (2.5,0.5*1.732) {$\bullet$};
        \node at (3.5,0.5*1.732) {$\bullet$};
        \node at (4.5,0.5*1.732) {$\bullet$};
		\draw (0.5, 0.5*1.732) -- (1.5, 0.5*1.732) -- (2.5, 0.5*1.732);
        \draw (2.5, 0.5*1.732) -- (3.5, 0.5*1.732);
		\draw (3.5, 0.5*1.732) -- (4.5, 0.5*1.732);
        \node at (1,1.732) {$\bullet$};
        \node at (2,1.732) {$\bullet$};
        \node at (3,1.732) {$\bullet$};
        \node at (4,1.732) {$\bullet$};
        \node at (5,1.732) {$\bullet$};
		\draw (1, 1.732) -- (2, 1.732) -- (3, 1.732) -- (4, 1.732) -- (5, 1.732);
        \node at (1.5,1.5*1.732) {$\bullet$};
        \node at (2.5,1.5*1.732) {$\bullet$};
        \node at (3.5,1.5*1.732) {$\bullet$};
        \node at (4.5,1.5*1.732) {$\bullet$};
        \node at (5.5,1.5*1.732) {$\bullet$};
		\draw (1.5, 1.5*1.732) -- (2.5, 1.5*1.732);
		\draw (2.5, 1.5*1.732) -- (3.5, 1.5*1.732) -- (4.5, 1.5*1.732) -- (5.5, 1.5*1.732);
        \node at (2,2*1.732) {$\bullet$};
        \node at (3,2*1.732) {$\bullet$};
        \node at (4,2*1.732) {$\bullet$};
        \node at (5,2*1.732) {$\bullet$};
        \node at (6,2*1.732) {$\bullet$};
        \draw (2,2*1.732) -- (3,2*1.732);
		\draw (3, 2*1.732) -- (4,2*1.732) -- (5, 2*1.732) -- (6,2*1.732);
		
		\draw (0,0) -- (0.5, 0.5*1.732) -- (1, 1.732) -- (1.5, 1.5*1.732);
        \draw (1.5, 1.5*1.732) -- (2, 2*1.732);
		\draw (1,0) -- (1.5, 0.5* 1.732) -- (2, 1.732) -- (2.5, 1.5*1.732);
        \draw (2.5, 1.5*1.732) -- (3, 2*1.732);
		\draw (2,0) -- (2.5, 0.5* 1.732) -- (3, 1.732) -- (3.5, 1.5*1.732) -- (4, 2*1.732);
		\draw (3,0) -- (3.5, 0.5* 1.732) -- (4, 1.732) -- (4.5, 1.5*1.732);
        \draw (4.5, 1.5*1.732) -- (5, 2*1.732);
		\draw (4,0) -- (4.5, 0.5* 1.732) -- (5, 1.732);
        \draw (5, 1.732) -- (5.5, 1.5*1.732);
		\draw (5.5, 1.5*1.732) -- (6, 2*1.732);
			
		\draw (1,0) -- (0.5, 0.5*1.732);
		\draw (2,0) -- (1.5, 0.5*1.732) -- (1,1.732);
		\draw (3,0) -- (2.5, 0.5*1.732);
		\draw (2.5, 0.5*1.732) -- (2, 1.732) -- (1.5, 1.5*1.732);
        \draw (4,0) -- (3.5, 0.5*1.732);
		\draw (3.5, 0.5*1.732) -- (3, 1.732) -- (2.5, 1.5*1.732) -- (2, 2*1.732);
		\draw (4.5, 0.5*1.732) -- (4, 1.732) -- (3.5, 1.5*1.732) -- (3, 2*1.732);
        \draw (5,1.732) -- (4.5, 1.5*1.732);
		\draw (4.5, 1.5*1.732) -- (4, 2*1.732);
        \draw (5.5, 1.5*1.732) -- (5, 2*1.732);
	\end{tikzpicture}
 \end{center}

    \caption{The 4-hive parallelogram}\label{hive-grid}
\end{figure}

In the next two subsections,  we give a hive model for the flagged skew Littlewood-Richardson coefficients.
\subsection{Skew hives}
 \label{sec:skewhive} 
The $(n+1) \times (n+1)$  array of nodes in figure \ref{hive-grid} is called the \emph{n-hive parallelogram}. Observe that the small rhombi \footnote{A rhombus with unit side length.} in the $n$-hive parallelogram are oriented in the following three different ways:
$$\text{Northeast (NE):} \;\;\; \begin{tikzpicture}
  [scale=0.4]
  \draw (0,0) -- (1,0);
  \draw (0,0) -- (0.5, 0.5*1.732);
  \draw (0.5, 0.5*1.732) -- (1.5, 0.5*1.732);
  \draw (1,0) -- (1.5,0.5*1.732);
  
\end{tikzpicture}   \; , \;\;
 \text{Southeast (SE):} \;\;\;   \begin{tikzpicture}
  [scale=0.4]
  \draw (0,0) -- (1,0);
  \draw (0,0) -- (-0.5, 0.5*1.732);
  \draw (-0.5, 0.5*1.732) -- (0.5, 0.5*1.732);
  \draw (1,0) -- (0.5,0.5*1.732);
  
\end{tikzpicture}    \quad \text{ and }
\; \text{Vertical: }
\begin{tikzpicture}
 [scale=0.3]
  \draw (0,0) -- (0.5,0.5*1.732);
  \draw (1,0) -- (0.5,0.5*1.732);
  \draw (0.5,-0.5*1.732) -- (0,0);
  \draw (0.5,-0.5*1.732) -- (1,0);
  
\end{tikzpicture} .$$
Let $\lambda \in \mathcal{P}[n]$. Following \cite{krvfpsac}, we define  $\bar{\lambda} = (0, \lambda_1, \lambda_1 + \lambda_2, \cdots, |\lambda|)$ and $\partial \lambda = (\lambda_2 - \lambda_1, \lambda_3- \lambda_2, \cdots, \lambda_n - \lambda_{n-1})$. Let $\lambda, \mu, \gamma, \nu \in \mathcal{P}[n]$ be such that $\gamma \subset \mu, \, \lambda \subset \nu$ and $|\lambda| +|\mu| = |\nu|+|\gamma|$. We define the \emph{skew hive polytope} $\SHive(\lambda, \mu, \gamma, \nu)$ as the set of all $\mathbb{R}$-labellings of the nodes of the $n$-hive parallelogram such that:
\begin{enumerate}
	\item The boundary labels of the left boundary (read top to bottom); bottom boundary (read left to right); top boundary (read left to right); right boundary (read top to bottom) are $\bar{\lambda}, |\lambda| + \bar{\mu}, \bar{\gamma} \text{ and } |\gamma|+\bar{\nu}$ respectively as in figure \ref{fig:hive-example}.
	\item The {\em contents} of all the small rhombi are non-negative. We recall that the content of a small rhombus is the sum of the labels on its obtuse-angled nodes minus the sum of the labels on its acute-angled nodes.
\end{enumerate}
We denote the set of integer points in $\SHive(\lambda, \mu, \gamma, \nu)$ by $\SHive_\mathbb{Z}(\lambda, \mu, \gamma, \nu)$.
\begin{figure}[h]
\centering
    \begin{tikzpicture}
    \node at (0,0) {$\bullet$};
    \node at (1,0) {$\bullet$};
    \node at (2,0) {$\bullet$};
    \node at (3,0) {$\bullet$};
    \node at (4,0) {$\bullet$};
    \node at (0.5,0.5*1.732) {$\bullet$};
    \node at (4.5,0.5*1.732) {$\bullet$};
    \node at (1,1.732) {$\bullet$};
    \node at (5,1.732) {$\bullet$};
    \node at (1.5,1.5*1.732) {$\bullet$};
    \node at (5.5,1.5*1.732) {$\bullet$};
    \node at (2,2*1.732) {$\bullet$};
    \node at (3,2*1.732) {$\bullet$};
    \node at (4,2*1.732) {$\bullet$};
    \node at (5,2*1.732) {$\bullet$};
    \node at (6,2*1.732) {$\bullet$};
    \node at (-0.3,-0.3) {$|\lambda|$};
    \node at (-0.3,0.5*1.732) {$\sum_{i=1}^{3}\lambda_i$};
    \node at (0.1,1.732) {$\sum_{i=1}^{2} \lambda_i$};
    \node at (1.1,1.5*1.732) {$\lambda_1$};
    \node at (2,2.19*1.732) {$0$};
    \node at (3,2.2*1.732) {$\gamma_1$};
    \node at (4,2.2*1.732) {$ \substack{\sum_{i=1}^{2}\gamma_i}$};
    \node at (5.3,2.2*1.732) {$ \substack{\sum_{i=1}^{3}\gamma_i}$};
    \node at (6.3,2.1*1.732) {$ \substack{|\gamma|}$};
    \node at (6.2,1.5*1.732) {$\substack{|\gamma|+\nu_1}$};
    \node at (6.1,1*1.732) {$\substack{|\gamma|+\sum_{i=1}^{2}\nu_i}$};
    \node at (5.5,0.5*1.732) {$\substack{|\gamma|+\sum_{i=1}^{3}\nu_i}$};
    \node at (1,-0.5) {$\substack{\mu_1 \\ + |\lambda|}$};
    \node at (2,-0.5) {$\substack{\sum_{i=1} ^2 \mu _i \\+ |\lambda|}$};
    \node at (3.3,-0.5) {$\substack{\sum_{i=1} ^3 \mu _i \\+ |\lambda|}$};
    \node at (5,-0.4) {$\substack{|\lambda|+|\mu |=|\gamma|+|\nu|}$};
    \node at (0.5,0.5*1.732) {$\bullet$};
    \node at (4.5,0.5*1.732) {$\bullet$};
    \node at (1,1.732) {$\bullet$};
    \node at (5,1.732) {$\bullet$};
    \node at (1.5,1.5*1.732) {$\bullet$};
    \node at (5.5,1.5*1.732) {$\bullet$};
    \node at (2,2*1.732) {$\bullet$};
    \node at (3,2*1.732) {$\bullet$};
    \node at (4,2*1.732) {$\bullet$};
    \node at (5,2*1.732) {$\bullet$};
    \node at (6,2*1.732) {$\bullet$};
    \draw (0,0) -- (1,0) -- (2,0) -- (3,0);
    \draw (3,0) -- (4,0);
    \draw (0.5, 0.5*1.732) -- (1.5, 0.5*1.732) -- (2.5, 0.5*1.732);
    \draw (2.5, 0.5*1.732) -- (3.5, 0.5*1.732);
    \draw (3.5, 0.5*1.732) -- (4.5, 0.5*1.732);
    \draw (1, 1.732) -- (2, 1.732) -- (3, 1.732) -- (4, 1.732) -- (5, 1.732);
    \draw (1.5, 1.5*1.732) -- (2.5, 1.5*1.732);
    \draw (2.5, 1.5*1.732) -- (3.5, 1.5*1.732) -- (4.5, 1.5*1.732) -- (5.5, 1.5*1.732);
    \draw (2,2*1.732) -- (3,2*1.732);
    \draw (3, 2*1.732) -- (4,2*1.732) -- (5, 2*1.732) -- (6,2*1.732);
    \draw (0,0) -- (0.5, 0.5*1.732) -- (1, 1.732) -- (1.5, 1.5*1.732);
    \draw (1.5, 1.5*1.732) -- (2, 2*1.732);
    \draw (1,0) -- (1.5, 0.5* 1.732) -- (2, 1.732) -- (2.5, 1.5*1.732);
    \draw (2.5, 1.5*1.732) -- (3, 2*1.732);
    \draw (2,0) -- (2.5, 0.5* 1.732) -- (3, 1.732) -- (3.5, 1.5*1.732) -- (4, 2*1.732);
    \draw (3,0) -- (3.5, 0.5* 1.732) -- (4, 1.732) -- (4.5, 1.5*1.732);
    \draw (4.5, 1.5*1.732) -- (5, 2*1.732);
    \draw (4,0) -- (4.5, 0.5* 1.732) -- (5, 1.732);
    \draw (5, 1.732) -- (5.5, 1.5*1.732);
    \draw (5.5, 1.5*1.732) -- (6, 2*1.732);
    \draw (1,0) -- (0.5, 0.5*1.732);
    \draw (2,0) -- (1.5, 0.5*1.732) -- (1,1.732);
    \draw (3,0) -- (2.5, 0.5*1.732);
    \draw (2.5, 0.5*1.732) -- (2, 1.732) -- (1.5, 1.5*1.732);
    \draw (4,0) -- (3.5, 0.5*1.732);
    \draw (3.5, 0.5*1.732) -- (3, 1.732) -- (2.5, 1.5*1.732) -- (2, 2*1.732);
    \draw (4.5, 0.5*1.732) -- (4, 1.732) -- (3.5, 1.5*1.732) -- (3, 2*1.732);
    \draw (5,1.732) -- (4.5, 1.5*1.732);
    \draw (4.5, 1.5*1.732) -- (4, 2*1.732);
    \draw (5.5, 1.5*1.732) -- (5, 2*1.732);
    \end{tikzpicture}
    \label{Hive Polytope}
    \begin{tikzpicture}
    \draw (0,-0.8) node {$\null$};
    \draw (2,2*1.732) node {$0$};
    \draw (3,2*1.732) node {$2$};
    \draw (4,2*1.732) node {$3$};
    \draw (5,2*1.732) node {$3$};
    \draw (6,2*1.732) node {$3$};
    \draw (1.5,1.5*1.732) node {$3$};
    \draw (2.5,1.5*1.732) node {$7$};
    \draw (3.5,1.5*1.732) node {$9$};
    \draw (4.5,1.5*1.732) node {$10$};
    \draw (5.5,1.5*1.732) node {$10$};
    \draw (1,1.732) node {$4$};
    \draw (2,1.732) node {$9$};
    \draw (3,1.732) node {$13$};
    \draw (4,1.732) node {$14$};
    \draw (5,1.732) node {$14$};
    \draw (0.5,0.5*1.732) node {$5$};
    \draw (1.5,0.5*1.732) node {$10$};
    \draw (2.5,0.5*1.732) node {$14$};
    \draw (3.5,0.5*1.732) node {$16$};
    \draw (4.5,0.5*1.732) node {$16$};
    \draw (0,0) node {$5$};
    \draw (1,0) node {$10$};
    \draw (2,0) node {$14$};
    \draw (3,0) node {$16$};
    \draw (4,0) node {$17$};
    \draw (0.25,0) -- (0.75,0);
    \draw (1.25,0) -- (1.75,0);
    \draw (2.25,0) -- (2.75,0);
    \draw (3.25,0) -- (3.75,0);
    \draw (0.75,0.5*1.732) -- (1.25,0.5*1.732);
    \draw (1.75,0.5*1.732) -- (2.25,0.5*1.732);
    \draw (2.75,0.5*1.732) -- (3.25,0.5*1.732);
    \draw (3.75,0.5*1.732) -- (4.25,0.5*1.732);
    \draw (1.25,1.732) -- (1.75,1.732);
    \draw (2.25,1.732) -- (2.75,1.732);
    \draw (3.25,1.732) -- (3.75,1.732);
    \draw (4.25,1.732) -- (4.75,1.732);
    \draw (1.75,1.5*1.732) -- (2.25,1.5*1.732);
    \draw (2.75,1.5*1.732) -- (3.25,1.5*1.732);
    \draw (3.75,1.5*1.732) -- (4.25,1.5*1.732);
    \draw (4.75,1.5*1.732) -- (5.25,1.5*1.732);
    \draw (2.25,2*1.732) -- (2.75,2*1.732);
    \draw (5.25,2*1.732) -- (5.75,2*1.732);
    \draw (3.25,2*1.732) -- (3.75,2*1.732);
    \draw (4.25,2*1.732) -- (4.75,2*1.732);
    \draw (0+0.15, 0+0.2598) -- (0.35, 0.6062);
    \draw (0.5+0.15, 0.5*1.732+0.2598) -- (0.5+0.35, 0.5*1.732 + 0.6062);
    \draw (1+0.15, 1.732+0.2598) -- (1+0.35, 1.732 + 0.6062);
    \draw (1.5+0.15, 1.5*1.732+0.2598) -- (1.5+0.35, 1.5*1.732 + 0.6062);
    \draw (1+0.15, 0+0.2598) -- (1+0.35, 0 + 0.6062);
    \draw (1.5+0.15, 0.5*1.732+0.2598) -- (1.5+0.35, 0.5*1.732 + 0.6062);
    \draw (2+0.15, 1.732+0.2598) -- (2+0.35, 1.732 + 0.6062);
    \draw (2.5+0.15, 1.5*1.732+0.2598) -- (2.5+0.35, 1.5*1.732 + 0.6062);
    \draw (2+0.15, 0+0.2598) -- (2+0.35, 0 + 0.6062);
    \draw (2.5+0.15, 0.5*1.732+0.2598) -- (2.5+0.35, 0.5*1.732 + 0.6062);
    \draw (3+0.15, 1.732+0.2598) -- (3+0.35, 1.732 + 0.6062);
    \draw (3.5+0.15, 1.5*1.732+0.2598) -- (3.5+0.35, 1.5*1.732 + 0.6062);
    \draw (3+0.15, 0+0.2598) -- (3+0.35, 0 + 0.6062);
    \draw (3.5+0.15, 0.5*1.732+0.2598) -- (3.5+0.35, 0.5*1.732 + 0.6062);
    \draw (4+0.15, 1.732+0.2598) -- (4+0.35, 1.732 + 0.6062);
    \draw (4.5+0.15, 1.5*1.732+0.2598) -- (4.5+0.35, 1.5*1.732 + 0.6062);
    \draw (4+0.15, 0+0.2598) -- (4+0.35, 0 + 0.6062);
    \draw (4.5+0.15, 0.5*1.732+0.2598) -- (4.5+0.35, 0.5*1.732 + 0.6062);
    \draw (5+0.15, 1.732+0.2598) -- (5+0.35, 1.732 + 0.6062);
    \draw (5.5+0.15, 1.5*1.732+0.2598) -- (5.5+0.35, 1.5*1.732 + 0.6062);
    \draw (1-0.15, 0+0.2598) -- (1-0.35, 0 + 0.6062);
    \draw (2-0.15, 0+0.2598) -- (2-0.35, 0 + 0.6062);
    \draw (3-0.15, 0+0.2598) -- (3-0.35, 0 + 0.6062);
    \draw (4-0.15, 0+0.2598) -- (4-0.35, 0 + 0.6062);
    \draw (1.5-0.15, 0.5*1.732+0.2598) -- (1.5-0.35, 0.5*1.732 + 0.6062);
    \draw (2.5-0.15, 0.5*1.732+0.2598) -- (2.5-0.35, 0.5*1.732 + 0.6062);
    \draw (3.5-0.15, 0.5*1.732+0.2598) -- (3.5-0.35, 0.5*1.732 + 0.6062);
    \draw (4.5-0.15, 0.5*1.732+0.2598) -- (4.5-0.35, 0.5*1.732 + 0.6062);
    \draw (2-0.15, 1.732+0.2598) -- (2-0.35, 1.732 + 0.6062);
    \draw (3-0.15, 1.732+0.2598) -- (3-0.35, 1.732 + 0.6062);
    \draw (4-0.15, 1.732+0.2598) -- (4-0.35, 1.732 + 0.6062);
    \draw (5-0.15, 1.732+0.2598) -- (5-0.35, 1.732 + 0.6062);
    \draw (2.5-0.15, 1.5*1.732+0.2598) -- (2.5-0.35, 1.5*1.732 + 0.6062);
    \draw (3.5-0.15, 1.5*1.732+0.2598) -- (3.5-0.35, 1.5*1.732 + 0.6062);
    \draw (4.5-0.15, 1.5*1.732+0.2598) -- (4.5-0.35, 1.5*1.732 + 0.6062);
    \draw (5.5-0.15, 1.5*1.732+0.2598) -- (5.5-0.35, 1.5*1.732 + 0.6062);
    \end{tikzpicture}
    \caption{(a) The boundary labels of skew hives in $ \SHive(\lambda, \mu, \gamma, \nu)$.
    (b) A skew hive in $\SHive_{\mathbb{Z}}(\lambda, \mu, \gamma, \nu) $ where $\lambda =(3,1,1,0)$; $\gamma=(2,1,0,0)$; $\mu=(5,4,2,1)$; $\nu=(7,4,2,1)$.}
    \label{fig:hive-example}
\end{figure}
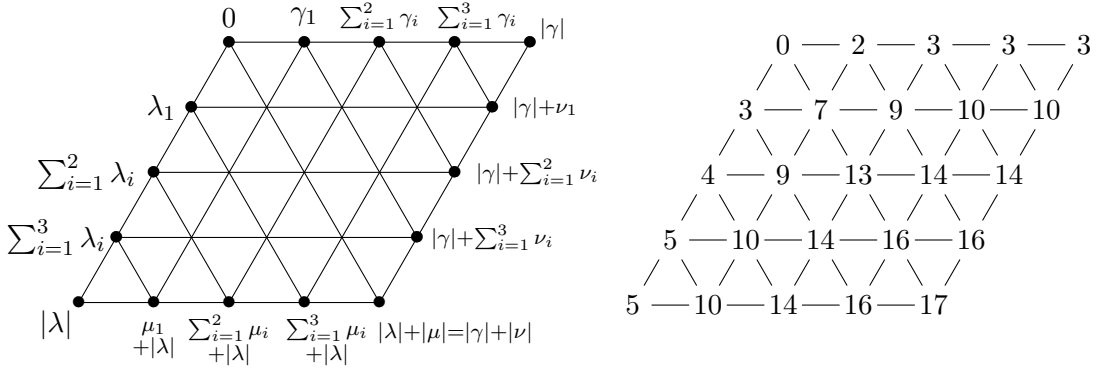
An element of $\SHive(\lambda, \mu, \gamma, \nu)$ is called a skew hive with boundary $(\lambda, \mu, \gamma, \nu)$. The  rows of a skew hive  $h \in \SHive(\lambda, \mu, \gamma, \nu)$ (from top to bottom) give a sequence of vectors $h_0, h_1, \cdots, h_n \in \mathbb{R}^{n+1}$. 
Consider the parallelogram array $\partial h$ with $(n+1)$ rows (and $n$ nodes in each row) whose rows (from top to bottom) are $\partial h_0, \partial h_1, \cdots, \partial h_n$. It follows directly (adapting the arguments of \cite{krvfpsac}) that the positivity of the contents of Northeast and Southeast rhombi in $h$ correspond to positivity (in the skew GT setting of \S\ref{sec:skewgt}) of $NE_{ij}$ and $ SE_{ij}$ of $\partial h$ respectively. Thus $\partial h $ is a skew GT pattern with top row $ \gamma $ and bottom row $ \mu$. It is elementary to check that the map $\partial: \SHive(\lambda, \mu, \gamma, \nu) \longrightarrow \GT(\mu / \gamma, n)$ is linear, injective and maps integer points to integer points.

\begin{theorem}
\label{hive-thrm}
For $\lambda, \mu, \gamma, \nu \in \mathcal{P}[n]$ such that $\gamma \subset \mu, \, \lambda \subset \nu$ and $|\lambda| +|\mu| = |\nu|+|\gamma|$, the map $\Upsilon \circ \partial$ is a bijection between $\SHive_{\mathbb{Z}}(\lambda, \mu, \gamma, \nu)$ \text{and } $Tab_{\lambda} ^{\nu}(\mu/\gamma, n) = \{T \in \Tab(\mu / \gamma, n): b_{T_{\lambda}^{0}} * b_T$ is a dominant word of weight $\nu$\}.    
\end{theorem}
\begin{proof}
We follow closely the proof of Proposition 4 in \cite{kushwaha2022saturation}.
Let $h \in \SHive_{\mathbb{Z}}(\lambda, \mu, \gamma, \nu)$. By the discussion in the previous para and lemma \ref{louck_lemma}, it is clear that $T=\Upsilon(\partial h) \in \Tab(\mu /\gamma ,n)$.\\
Now we will show that $ T \in Tab_{\lambda} ^{\nu}(\mu/\gamma, n)$.$\text{ }$Let $\partial h=(x_{ij}). \text{ Then } x_{ij}=h_{ij}-h_{i(j-1)}$. So the number of times $i$ appears in row $j$ of  $T=x_{ij}-x_{(i-1)j}$.
Now we have to prove that $ b_{T_{\lambda}^{0}} * b_T$ is a dominant word of weight $\nu$. Let $b_T= b_{T_1}*b_{T_2}*\cdots *b_{T_n}$ where $b_{T_k}$ is the reverse reading word of the $k$-th row of $T$. Also, let $ N_{ik}$  be the number of times $i$ appears in the word  $ b_{T_0}*b_{T_1}* \cdots * b_{T_k}$ (with $b_{T_0} =b_{T_{\lambda}^{0}}$). Then by definition, $ b_{T_{\lambda}^{0}} * b_T$ is dominant iff  $N_{ik} \geq N_{(i+1)(k+1)} \text{ for all } 1 \leq i \leq n, \text{ } 0 \leq k \leq n$.
We get
$$N_{ik} = \lambda_i + (x_{i1}-x_{(i-1)1})+\cdots + (x_{ik}-x_{(i-1)k})
= h_{ik}-h_{(i-1)k} \text{ (in terms of $h$) }$$
So $N_{ik} - N_{(i+1)(k+1)} = h_{ik}+h_{i(k+1)}-h_{(i-1)k}-h_{(i+1)(k+1)} \geq 0 $ by the corresponding vertical rhombus inequality in $h$.
Now $\text{ the number of times $i$ occurs in }b_{T_{\lambda}^{0}} * b_T$ is
$$\lambda_i + (x_{i1}-x_{(i-1)1})+\cdots + (x_{in}-x_{(i-1)n})=h_{in}-h_{(i-1)n} =\nu_i $$
So $ b_{T_{\lambda}^{0}} * b_T$ is a dominant word of weight $\nu$. Thus $T \in Tab_{\lambda} ^{\nu}(\mu/\gamma, n)$.

Now, we will show $\Upsilon \circ \partial$ is injective. It suffices to show that  $\partial$ is injective.
Let $ h=(h_0,h_1,\cdots ,h_n), \, h'=(h'_0,h'_1,\cdots ,h'_n) \in \Hive_\mathbb{Z}(\lambda, \mu, \gamma, \nu)$ where $h_i,h'_i \in \mathbb{Z}^{n+1} \text{ and } \partial h = \partial h'$. Also, let $h_i=(h_{i0},h_{i1},\cdots,h_{in}) \text{ and } h'_i=(h'_{i0},h'_{i1},\cdots,h'_{in})$.
We will show that $h_{ik}=h'_{ik}$ for all $i$ by induction on $k$. Clearly, $h_{00}=h'_{00}=0 \text{ and } h_{i0}= \sum_{j=1}^{i} \lambda_j = h'_{i0}$. So for $k=0$ we get $h_{ik}=h'_{ik}$ for all $i$. Let $ k \geq 1 \text{ and } h_{i(k-1)}=h'_{i(k-1)}$ for all $i$. Since $\partial h_i = \partial h'_i.$ So $h_{ik}-h_{i(k-1)} = h'_{ik}-h'_{i(k-1)} $. Thus $h_{ik}=h'_{ik}$ for all $i$. Hence $ h=h' $.
Therefore $\partial $ is injective.

Next we will prove $ \Upsilon \circ \partial $ is surjective.
Let $ S \in Tab_{\lambda} ^{\nu}(\mu/\gamma, n)$ and  $\Upsilon^{-1}(S)= (s_{ij})$ $\in \GT_{\mathbb{Z}}(\mu / \gamma, n) $ where $s_{0j}= \mu_j, \, s_{ij}= \mu_j + $ the numbers of entries $\leq i$ in the $j^{th}$ row of $S$ $(1 \leq i,j \leq n) $. Define $ \lambda_{0}=0, h_{i0}=\sum_{k=0}^{i} \lambda_k, h_{ij}=h_{i0}+\sum_{k=1}^{j} s_{ik} \text{ }(0 \leq i \leq n, \text{ } 1 \leq j \leq n) \text{ and } h_i=(h_{i0},h_{i1},\cdots,h_{in})$ for all $i$. Consider  the $n$-hive parallelogram $h$ whose rows (from top to bottom) are $h_0, h_1, \cdots, h_n$. We claim that $h \in \SHive_{\mathbb{Z}}(\lambda, \mu, \gamma, \nu)$. The Northeast and Southeast rhombi inequalities of $h$ hold since $ S \in \Tab(\mu /\gamma ,n)$ and the vertical rhombi inequalities of $h$  hold because $ b_{T_{\lambda}^{0}} * b_S $ is dominant. The boundary labels of $h$ can be easily seen to be $\bar{\lambda}$ (left edge, top to bottom), $|\lambda| + \bar{\mu}$ (bottom edge, left to right), $\bar{\gamma}$ (top edge, left to right) and $|\gamma|+\bar{\nu}$ (right edge, top to bottom) respectively. So $ h \in \SHive_\mathbb{Z}(\lambda, \mu, \gamma, \nu)$ and $ (\Upsilon \circ \partial)^{-1} (S)=h $. Thus $\Upsilon \circ \partial $ is surjective.
\end{proof}

\subsection{Flagged skew hives}
Given a flag $\Phi$, consider the face of the skew hive polytope defined by
 $\SHive(\lambda, \mu, \gamma, \nu, \Phi) := \partial ^{-1} ( \GT(\mu / \gamma, \Phi))$. We call this the \emph{flagged skew hive polytope}. We let  $\SHive_{\mathbb{Z}}(\lambda, \mu, \gamma, \nu, \Phi)$ denote the set of integer points in this polytope. It is elementary to observe that 
$\SHive_{\mathbb{Z}}(\lambda, \mu, \gamma, \nu, \Phi) = \partial ^{-1} ( \GT_{\mathbb{Z}}(\mu / \gamma, \Phi))$.
Lemma~\ref{lem:flags} and Theorem~\ref{hive-thrm} together give us the second part of theorem \ref{first result}:
 \begin{theorem}
   The map $\Upsilon \circ \partial$ restricts to a bijection between $\SHive_{\mathbb{Z}}(\lambda, \mu, \gamma, \nu, \Phi)$ and $Tab_{\lambda} ^{\nu}(\mu /\gamma, \Phi)$. Then $c_{\lambda, \,\mu/\gamma} ^{\,\nu} (\Phi)=|\SHive_{\mathbb{Z}}(\lambda, \mu, \gamma, \nu, \Phi)|=|Tab_{\lambda} ^{\nu}(\mu /\gamma, \Phi)|$.    
 \end{theorem}
The NE rhombi of the $n$-hive parallelogram are labelled $R_{ij}$ with $1 \leq i,j \leq n$ as shown in the example in Figure~\ref{left bottom}. To the flag $\Phi$, we associate the set $I= \bigcup_{i=1}^n \{(\Phi_i+1,i),\cdots,(n,i)\}$. Then $ R(\Phi)= \bigcup_{(i,j) \in I} \{ R_{ij} \}$ forms a bottom-left-justified region of NE  rhombi (shown in purple in Figure~\ref{left bottom}). It is easy to see from equation~\eqref{eq:flaggt} that  $\SHive_{\mathbb{Z}}(\lambda, \mu, \gamma, \nu, \Phi)$ is obtained from  $\SHive_{\mathbb{Z}}(\lambda, \mu, \gamma, \nu)$ by imposing the condition that all rhombi in $R(\Phi)$ are flat. For example, the hive in figure \ref{fig:hive-example}(b) is such that every NE oriented rhombus in $R(\Phi)$ is flat where $\Phi$ = (2,2,3,4).
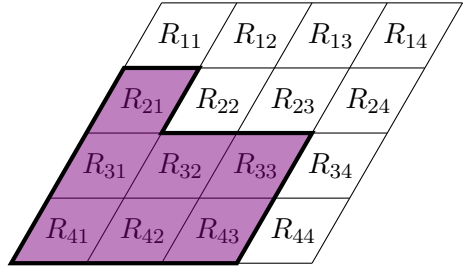
\begin{figure}[t]
    \centering
    \begin{center}
\begin{tikzpicture}
		\draw (0,0) -- (1,0) -- (2,0) -- (3,0);
		\draw (3,0) -- (4,0);
		\draw (0.5, 0.5*1.732) -- (1.5, 0.5*1.732) -- (2.5, 0.5*1.732);
		\draw (2.5, 0.5*1.732) -- (3.5, 0.5*1.732);
		\draw (3.5, 0.5*1.732) -- (4.5, 0.5*1.732);
		\draw (1, 1.732) -- (2, 1.732) -- (3, 1.732) -- (4, 1.732) -- (5, 1.732);
		\draw (1.5, 1.5*1.732) -- (2.5, 1.5*1.732);
		\draw (2.5, 1.5*1.732) -- (3.5, 1.5*1.732) -- (4.5, 1.5*1.732) -- (5.5, 1.5*1.732);
		\draw (2,2*1.732) -- (3,2*1.732);
		\draw (3, 2*1.732) -- (4,2*1.732) -- (5, 2*1.732) -- (6,2*1.732);
		\draw (0,0) -- (0.5, 0.5*1.732) -- (1, 1.732) -- (1.5, 1.5*1.732);
		\draw (1.5, 1.5*1.732) -- (2, 2*1.732);
		\draw (1,0) -- (1.5, 0.5* 1.732) -- (2, 1.732) -- (2.5, 1.5*1.732);
		\draw (2.5, 1.5*1.732) -- (3, 2*1.732);
		\draw (2,0) -- (2.5, 0.5* 1.732) -- (3, 1.732) -- (3.5, 1.5*1.732) -- (4, 2*1.732);
		\draw (3,0) -- (3.5, 0.5* 1.732) -- (4, 1.732) -- (4.5, 1.5*1.732);
		\draw (4.5, 1.5*1.732) -- (5, 2*1.732);
		\draw (4,0) -- (4.5, 0.5* 1.732) -- (5, 1.732);
		\draw (5, 1.732) -- (5.5, 1.5*1.732);
		\draw (5.5, 1.5*1.732) -- (6, 2*1.732);
		
  \node at (0.75, 0.25*1.732) {$R_{41}$};
  \node at (1.75, 0.25*1.732) {$R_{42}$};
  \node at (2.75, 0.25*1.732) {$R_{43}$};
  \node at (3.75, 0.25*1.732) {$R_{44}$};
  \node at (1.25, 0.75*1.732) {$R_{31}$};
  \node at (2.25, 0.75*1.732) {$R_{32}$};
  \node at (3.25, 0.75*1.732) {$R_{33}$};
  \node at (4.25, 0.75*1.732) {$R_{34}$};
  \node at (1.75, 1.25*1.732) {$R_{21}$};
  \node at (2.75, 1.25*1.732) {$R_{22}$};
  \node at (3.75, 1.25*1.732) {$R_{23}$};
  \node at (4.75, 1.25*1.732) {$R_{24}$};
  \node at (2.25, 1.75*1.732) {$R_{11}$};
  \node at (3.25, 1.75*1.732) {$R_{12}$};
  \node at (4.25, 1.75*1.732) {$R_{13}$};
  \node at (5.25, 1.75*1.732) {$R_{14}$};
  \draw [ultra thick, draw=black, fill=violet, fill opacity=0.5]
       (0,0) -- (1.5,1.5*1.732) -- (2.5,1.5*1.732) -- (2,1*1.732) -- (4,1*1.732) --  
       (3,0) -- cycle;
\end{tikzpicture}        
    \end{center}
    \caption{Labelling of NE oriented rhombi in 4-hive parallegram and the shaded region is a typical configuration of $ R(\Phi)$ (for $ \Phi=(1,2,2,4)$)}\label{left bottom}
\end{figure}

\section{Flagged Skew LR coefficients are $w$-refined LR coefficients}
In this section we prove that any skew hive polytope is affinely isomorphic to some \emph{hive polytope} (albeit in twice as many ambient dimensions). Moreover, this isomorphism maps the flagged skew hive polytope to a hive Kogan face corresponding to some $312$-avoiding permutation \cite[\S 2.4]{krvfpsac} and preserves the integral points.

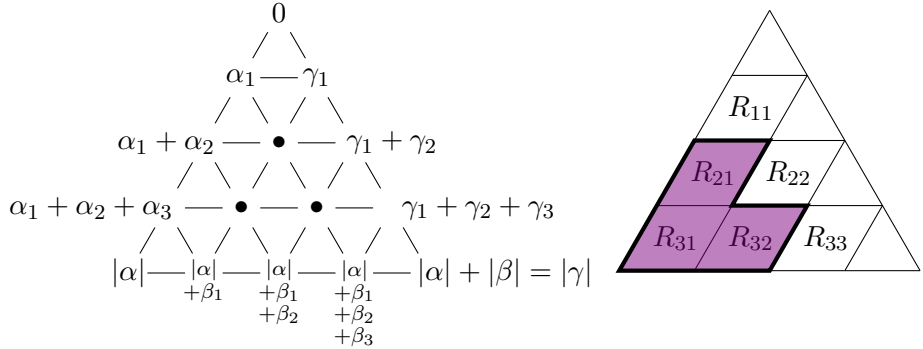
\begin{figure}[h]
    \begin{tikzpicture}
    \draw (3.5,3.5*1.732) node {$0$};
    \draw (3,3*1.732) node {$\alpha_1$};
    \draw (4,3*1.732) node {${\gamma}_1 $};
    \draw (5,2.5*1.732) node {${\gamma}_1 + \gamma_2$};
    \draw (3.5,2.5*1.732) node {$\bullet$};
    \draw (2,2.5*1.732) node {$\alpha_1 + \alpha_2$};
    \draw (1,2*1.732) node {$\alpha_1 + \alpha_2 + \alpha_3$};
    \draw (3,2*1.732) node {$\bullet$};
    \draw (4,2*1.732) node {$\bullet$};
    \draw (6.15,2*1.732) node {$\gamma_1+\gamma_2+\gamma_3$};
    \draw (1.5,1.5*1.732) node {$|\alpha|$};
    \draw (2.5,1.5*1.732-0.1) node {$\substack{|\alpha|\\+\beta_1}$};
    \draw (3.5,1.5*1.732-0.25) node {$\substack{|\alpha|\\+\beta_1\\+\beta_2}$};
    \draw (4.5,1.5*1.732-0.4) node {$\substack{|\alpha|\\+\beta_1\\+\beta_2\\+\beta_3}$};
    \draw (6.5,1.5*1.732) node {$|\alpha|+|\beta| = |\gamma|$};
    \draw (1.75,1.5*1.732) -- (2.25,1.5*1.732);
    \draw (2.75,1.5*1.732) -- (3.25,1.5*1.732);
    \draw (3.75,1.5*1.732) -- (4.25,1.5*1.732);
    \draw (4.75,1.5*1.732) -- (5.25,1.5*1.732);
    \draw (2.75,2.5*1.732) -- (3.25,2.5*1.732);
    \draw (3.75,2.5*1.732) -- (4.25,2.5*1.732);
    \draw (3.25,3*1.732) -- (3.75,3*1.732);
    \draw (2.25,2*1.732) -- (2.75,2*1.732);
    \draw (3.25,2*1.732) -- (3.75,2*1.732);
    \draw (4.25,2*1.732) -- (4.75,2*1.732);
    \draw (1.5+0.15, 1.5*1.732+0.2598) -- (1.5+0.35, 1.5*1.732 + 0.6062);
    \draw (2+0.15, 2*1.732+0.2598) -- (2+0.35, 2*1.732 + 0.6062);
    \draw (2.5+0.15, 2.5*1.732+0.2598) -- (2.5+0.35, 2.5*1.732 + 0.6062);
    \draw (3+0.15, 3*1.732+0.2598) -- (3+0.35, 3*1.732 + 0.6062);
    \draw (2.5+0.15, 1.5*1.732+0.2598) -- (2.5+0.35, 1.5*1.732 + 0.6062);
    \draw (3+0.15, 2*1.732+0.2598) -- (3+0.35, 2*1.732 + 0.6062);
    \draw (3.5+0.15, 2.5*1.732+0.2598) -- (3.5+0.35, 2.5*1.732 + 0.6062);
    \draw (3.5+0.15, 1.5*1.732+0.2598) -- (3.5+0.35, 1.5*1.732 + 0.6062);
    \draw (4+0.15, 2*1.732+0.2598) -- (4+0.35, 2*1.732 + 0.6062);
    \draw (4.5+0.15, 1.5*1.732+0.2598) -- (4.5+0.35, 1.5*1.732 + 0.6062);
    \draw (2.5-0.15, 1.5*1.732+0.2598) -- (2.5-0.35, 1.5*1.732 + 0.6062);
    \draw (3.5-0.15, 1.5*1.732+0.2598) -- (3.5-0.35, 1.5*1.732 + 0.6062);
    \draw (4.5-0.15, 1.5*1.732+0.2598) -- (4.5-0.35, 1.5*1.732 + 0.6062);
    \draw (5.5-0.15, 1.5*1.732+0.2598) -- (5.5-0.35, 1.5*1.732 + 0.6062);
    \draw (3-0.15, 2*1.732+0.2598) -- (3-0.35, 2*1.732 + 0.6062);
    \draw (4-0.15, 2*1.732+0.2598) -- (4-0.35, 2*1.732 + 0.6062);
    \draw (5-0.15, 2*1.732+0.2598) -- (5-0.35, 2*1.732 + 0.6062);
    \draw (3.5-0.15, 2.5*1.732+0.2598) -- (3.5-0.35, 2.5*1.732 + 0.6062);
    \draw (4.5-0.15, 2.5*1.732+0.2598) -- (4.5-0.35, 2.5*1.732 + 0.6062);
    \draw (4-0.15, 3*1.732+0.2598) -- (4-0.35, 3*1.732 + 0.6062);
    \end{tikzpicture}
    \begin{tikzpicture}
		\draw (0,0) -- (1,0) -- (2,0) -- (3,0);
		\draw (3,0) -- (4,0);
		\draw (0.5, 0.5*1.732) -- (1.5, 0.5*1.732) -- (2.5, 0.5*1.732);
		\draw (2.5, 0.5*1.732) -- (3.5, 0.5*1.732);
		\draw (1, 1.732) -- (2, 1.732) -- (3, 1.732);
		\draw (1.5, 1.5*1.732) -- (2.5, 1.5*1.732);
		\draw (0,0) -- (0.5, 0.5*1.732) -- (1, 1.732) -- (1.5, 1.5*1.732);
		\draw (1.5, 1.5*1.732) -- (2, 2*1.732);
		\draw (1,0) -- (1.5, 0.5* 1.732) -- (2, 1.732) -- (2.5, 1.5*1.732);
		\draw (2,0) -- (2.5, 0.5* 1.732) -- (3, 1.732);
		\draw (3,0) -- (3.5, 0.5* 1.732);
		\draw (4,0) -- (3.5, 0.5* 1.732) -- (3, 1.732) -- (2.5,1.5*1.732) -- (2, 2*1.732);
		
  \node at (0.75, 0.25*1.732) {$R_{31}$};
  \node at (1.75, 0.25*1.732) {$R_{32}$};
  \node at (2.75, 0.25*1.732) {$R_{33}$};
  \node at (1.25, 0.75*1.732) {$R_{21}$};
  \node at (2.25, 0.75*1.732) {$R_{22}$};
  \node at (1.75, 1.25*1.732) {$R_{11}$};
  \draw [ultra thick, draw=black, fill=violet, fill opacity=0.5]
       (0,0) -- (1,1*1.732) -- (2,1.732) -- (2,1*1.732) -- (1.5,0.5*1.732) -- (2.5,0.5*1.732) -- (2,0) -- cycle;
       \node at (0,-1) {$\null$};
\end{tikzpicture}        
    \caption{(a) A $4$-hive with boundary $(\alpha, \beta, \gamma)$ (b) Labelling of NE oriented rhombi. The shaded region is a typical configuration of $R(\Phi)$, shown here for $\Phi = (2,3,4,4)$} \label{tri-hives}
\end{figure}
Given partitions $\alpha, \beta, \gamma \in \mathcal{P}[n]$ such that $|\alpha| + |\beta| = |\gamma|$, the hive polytope $\Hive(\alpha, \beta, \gamma)$ is the set of all $\mathbb{R}$-labellings of an $(n+1)$-triangular array of nodes such that
\begin{enumerate}
    \item the boundary labels are given by $(0, \alpha_1, \alpha_1 + \alpha_2, \cdots, |\alpha|, |\alpha| + \beta_1, |\alpha|+\beta_1 + \beta_2, \cdots, |\alpha|+|\beta|, \gamma_1+\cdots+\gamma_{n-1}, \cdots, \gamma_1+\gamma_2, \gamma_1)$, reading the nodes anti-clockwise beginning from the topmost node as in figure \ref{tri-hives}.
    \item the contents of all the small rhombi are non-negative.
\end{enumerate}
The horizontal strings of nodes (``rows'') of the triangular array are termed the zeroth row, first row, second row, etc starting from the top.
Consider the labelling of the NE oriented small rhombi as shown in the example in figure \ref{tri-hives}. Given a flag $\Phi$, consider the set of NE rhombi $R(\Phi) = \{R_{ij}| \,n > i \geq \Phi_j\}$. We define the face $\Hive(\alpha, \beta, \gamma, \Phi)$ of the polytope $\Hive(\alpha, \beta, \gamma)$ as the collection of those hives in which all the rhombi in $R(\Phi)$ are flat. As we vary $\Phi$, these run over the {\em hive Kogan faces} corresponding to $312$-avoiding permutations, in the terminology of \cite[\S 2.4]{krvfpsac}. Refer \cite{krvfpsac} for more details.
\begin{lemma}\label{hive_singleton}
    Let $\alpha, \beta, \gamma \in \mathbb R _{+} ^n$ be weakly decreasing sequences such that either $\alpha$ or $\beta$ is a constant sequence. Then $\Hive(\alpha, \beta, \gamma)$ is either empty or a singleton set. The latter is true if and only if $\alpha+\beta = \gamma$. 
\end{lemma}
\begin{proof} Supposing first that $\beta=(\beta_1, \beta_2, \cdots, \beta_n)$ is constant, i.e., $\beta_i=b$ (say) for all $i$. If $\Hive(\alpha, \beta, \gamma)$ is non-empty, let $h$ be an element. Define $\kappa = \gamma - \alpha \in \mathbb R ^n$. The labels along the left and right edges of $h$ are $\sum_{i=1}^k \alpha_i$ and $\sum_{i=1}^k \gamma_i$ for $0 \leq k \leq n$. Now consider the following types of trapezia formed by two overlapping unit rhombi, one NE and the other SE.
    \begin{center}
    \begin{tikzpicture}
    \draw (2.5,1.5*1.732) node {$d$};
    \draw (3.5,1.5*1.732) node {$e$};
    \draw (2,1.732) node {$c$};
    \draw (3,1.732) node {$b$};
    \draw (4,1.732) node {$a$};
    \draw (2.25,1.732) -- (2.75,1.732);
    \draw (3.25,1.732) -- (3.75,1.732);
    \draw (2.75,1.5*1.732) -- (3.25,1.5*1.732);
    \draw (2+0.15, 1.732+0.2598) -- (2+0.35, 1.732 + 0.6062);
    \draw (3+0.15, 1.732+0.2598) -- (3+0.35, 1.732 + 0.6062);
    \draw (3-0.15, 1.732+0.2598) -- (3-0.35, 1.732 + 0.6062);
    \draw (4-0.15, 1.732+0.2598) -- (4-0.35, 1.732 + 0.6062);
    \end{tikzpicture}
    \end{center}
From the two rhombus inequalities in this picture, we conclude that if $a-b = b-c$, then each is also equal to $e-d$. Now $\beta_i=b$ for all $i$ implies that the successive differences of labels on the bottom (i.e., $n^{th}$) row of $h$ are all equal to $b$. The observation above means that the successive differences of labels on the $(n-1)^{th}$ row are also all $b$. We proceed by induction, moving up the hive triangle, to conclude that the successive differences of labels along every row of $h$ is equal to $b$. In particular, all labels of $h$ are uniquely determined from those on  left boundary alone. An additional compatibility condition arises from summing the differences in each row - this gives $\sum_{i=1}^k \kappa_i = kb$ for each $k$, or equivalently that $\kappa = \beta$ as claimed. 

The proof for the case that $\alpha$ is constant is similar. One considers instead the trapezia of the form:
\begin{center}
    \begin{tikzpicture}
    \draw (3,2*1.732) node {$a$};
    \draw (2.5,1.5*1.732) node {$b$};
    \draw (3.5,1.5*1.732) node {$e$};
    \draw (2,1.732) node {$c$};
    \draw (3,1.732) node {$d$};
    \draw (2.25,1.732) -- (2.75,1.732);
    \draw (2.75,1.5*1.732) -- (3.25,1.5*1.732);
    \draw (2+0.15, 1.732+0.2598) -- (2+0.35, 1.732 + 0.6062);
    \draw (2.5+0.15, 1.5*1.732+0.2598) -- (2.5+0.35, 1.5*1.732 + 0.6062);
    \draw (3+0.15, 1.732+0.2598) -- (3+0.35, 1.732 + 0.6062);
    \draw (3-0.15, 1.732+0.2598) -- (3-0.35, 1.732 + 0.6062);
    \draw (3.5-0.15, 1.5*1.732+0.2598) -- (3.5-0.35, 1.5*1.732 + 0.6062);
    \end{tikzpicture}
    \end{center}
\end{proof}
\begin{proposition}
    Let $\lambda, \mu, \nu, \gamma \in \mathcal{P}[n]$ be such that $\gamma \subset \mu, \, \lambda \subset \nu$ and $|\lambda| +|\mu| = |\nu|+|\gamma|$. Define  $\Tilde{\lambda}, \Tilde{\mu}, \Tilde{\nu} \in \mathcal{P}[2n]$ by $\Tilde{\lambda} = (\nu_1, \cdots, \nu_1, \lambda_1, \cdots,  \lambda_n), \, \Tilde{\mu} = (\mu_1, \mu_2, \cdots, \mu_n, 0,0,\cdots, 0)$ and $\Tilde{\nu} = (\nu_1 + \gamma_1, \cdots, \nu_1 + \gamma_n, \nu_1, \cdots, \nu_n)$. Then there exists an affine linear isomorphism between $\SHive(\lambda, \mu, \gamma, \nu)$ and $\Hive(\Tilde{\lambda}, \Tilde{\mu}, \Tilde{\nu})$.
    Moreover, the isomorphism preserves integral points and maps $\SHive(\lambda, \mu, \gamma, \nu, \Phi)$ onto $\Hive(\Tilde{\lambda}, \Tilde{\mu}, \Tilde{\nu}, \Tilde{\Phi})$, where $\Tilde{\Phi} = (\Phi_1 +n, \cdots, \Phi_n +n)$.
\end{proposition}
\begin{proof}
For $h \in \SHive(\lambda, \mu, \gamma, \nu)$ we describe $\psi(h)$ as follows (see figure \ref{par_to_tri} for a representative example):
\begin{enumerate}
    \item  $\psi(h)$ is a labelling of the $(2n+1)$- triangular array, with boundary labels coinciding with those of hives in  $\Hive(\Tilde{\lambda}, \Tilde{\mu}, \Tilde{\nu})$.
    \item The bottom-left-justified $(n+1) \times (n+1)$ parallelogram in $\psi(h)$ (white, with a red border in figure \ref{par_to_tri}) is labelled by $h + n \cdot \nu_1$, i.e., the labels of the nodes of $h$ are translated by the constant $n \cdot \nu_1$.
    \item The labels of the top $n$ rows of $\psi(h)$ (highlighted in yellow in figure \ref{par_to_tri}) are determined by the boundary conditions of $\Hive(\Tilde{\lambda}, \Tilde{\mu}, \Tilde{\nu})$ and the choice of parallelogram labels in (2) above. This follows from lemma \ref{hive_singleton}, using the fact that the first $n$ components of $\Tilde{\lambda}$ are equal 
    \item Again, using (2) and  the fact that the last $n$ components of $\Tilde{\mu}$ are equal, lemma \ref{hive_singleton} implies that the labels of the bottom-right-justified $(n+1)$-triangular subarray in
      $\psi(h)$ (highlighted in blue in figure \ref{par_to_tri}) are determined by the boundary conditions of $\Hive(\Tilde{\lambda}, \Tilde{\mu}, \Tilde{\nu})$.
\end{enumerate}
\begin{figure}[h]
    \begin{tikzpicture}
    \draw (2,2*1.732) node {$0$};
    \draw (3,2*1.732) node {$a$};
    \draw (4,2*1.732) node {$b$};
    \draw (5,2*1.732) node {$c$};
    \draw (1.5,1.5*1.732) node {$d$};
    \draw (2.5,1.5*1.732) node {$e$};
    \draw (3.5,1.5*1.732) node {$f$};
    \draw (4.5,1.5*1.732) node {$g$};
    \draw (1,1.732) node {$h$};
    \draw (2,1.732) node {$i$};
    \draw (3,1.732) node {$j$};
    \draw (4,1.732) node {$k$};
    \draw (0.5,0.5*1.732) node {$l$};
    \draw (1.5,0.5*1.732) node {$m$};
    \draw (2.5,0.5*1.732) node {$n$};
    \draw (3.5,0.5*1.732) node {$o$};
    \draw[red, thick] (0.75,0.5*1.732) -- (1.25,0.5*1.732);
    \draw[red, thick] (1.75,0.5*1.732) -- (2.25,0.5*1.732);
    \draw[red, thick] (2.75,0.5*1.732) -- (3.25,0.5*1.732);
    \draw (1.25,1.732) -- (1.75,1.732);
    \draw (2.25,1.732) -- (2.75,1.732);
    \draw (3.25,1.732) -- (3.75,1.732);
    \draw (1.75,1.5*1.732) -- (2.25,1.5*1.732);
    \draw (2.75,1.5*1.732) -- (3.25,1.5*1.732);
    \draw (3.75,1.5*1.732) -- (4.25,1.5*1.732);
    \draw[red, thick] (2.25,2*1.732) -- (2.75,2*1.732);
    \draw[red, thick] (3.25,2*1.732) -- (3.75,2*1.732);
    \draw[red, thick] (4.25,2*1.732) -- (4.75,2*1.732);
    \draw[red, thick] (0.5+0.15, 0.5*1.732+0.2598) -- (0.5+0.35, 0.5*1.732 + 0.6062);
    \draw[red, thick] (1+0.15, 1.732+0.2598) -- (1+0.35, 1.732 + 0.6062);
    \draw[red, thick] (1.5+0.15, 1.5*1.732+0.2598) -- (1.5+0.35, 1.5*1.732 + 0.6062);
    \draw (1.5+0.15, 0.5*1.732+0.2598) -- (1.5+0.35, 0.5*1.732 + 0.6062);
    \draw (2+0.15, 1.732+0.2598) -- (2+0.35, 1.732 + 0.6062);
    \draw (2.5+0.15, 1.5*1.732+0.2598) -- (2.5+0.35, 1.5*1.732 + 0.6062);
    \draw (2.5+0.15, 0.5*1.732+0.2598) -- (2.5+0.35, 0.5*1.732 + 0.6062);
    \draw (3+0.15, 1.732+0.2598) -- (3+0.35, 1.732 + 0.6062);
    \draw (3.5+0.15, 1.5*1.732+0.2598) -- (3.5+0.35, 1.5*1.732 + 0.6062);
    \draw[red, thick] (3.5+0.15, 0.5*1.732+0.2598) -- (3.5+0.35, 0.5*1.732 + 0.6062);
    \draw[red, thick] (4+0.15, 1.732+0.2598) -- (4+0.35, 1.732 + 0.6062);
    \draw[red, thick] (4.5+0.15, 1.5*1.732+0.2598) -- (4.5+0.35, 1.5*1.732 + 0.6062);
    \draw (1.5-0.15, 0.5*1.732+0.2598) -- (1.5-0.35, 0.5*1.732 + 0.6062);
    \draw (2.5-0.15, 0.5*1.732+0.2598) -- (2.5-0.35, 0.5*1.732 + 0.6062);
    \draw (3.5-0.15, 0.5*1.732+0.2598) -- (3.5-0.35, 0.5*1.732 + 0.6062);
    \draw (2-0.15, 1.732+0.2598) -- (2-0.35, 1.732 + 0.6062);
    \draw (3-0.15, 1.732+0.2598) -- (3-0.35, 1.732 + 0.6062);
    \draw (4-0.15, 1.732+0.2598) -- (4-0.35, 1.732 + 0.6062);
    \draw (2.5-0.15, 1.5*1.732+0.2598) -- (2.5-0.35, 1.5*1.732 + 0.6062);
    \draw (3.5-0.15, 1.5*1.732+0.2598) -- (3.5-0.35, 1.5*1.732 + 0.6062);
    \draw (4.5-0.15, 1.5*1.732+0.2598) -- (4.5-0.35, 1.5*1.732 + 0.6062);
    \end{tikzpicture}
    \begin{tikzpicture}
    \draw (0,0) node {$\null$};
    \draw[|->] (1,2) -- (2,2);
    \draw (1.45,2.3) node {$\psi$};
    \end{tikzpicture}
    \begin{tikzpicture}
    \draw [fill=yellow, draw=none, fill opacity=0.5]
       (2,2*1.732) -- (3.5,3.5*1.732) -- (5,2*1.732) -- cycle;
    \draw [fill=blue, draw=none, fill opacity=0.2]
       (5,2*1.732) -- (3.5 ,0.5*1.732) -- (6.5,0.5*1.732) -- cycle;
    \draw (3.5,3.5*1.732) node {$0$};
    \draw (3,3*1.732) node {$\nu_1$};
    \draw (4+0.5,3*1.732) node {$\nu_1 + {\gamma}_1 $};
    \draw (5+0.5,2.5*1.732) node {$2\nu_1 + {\gamma}_1 + \gamma_2$};
    \draw (3.5,2.5*1.732) node {$q$};
    \draw (2.5,2.5*1.732) node {$2 \nu_1$};
    \draw (2,2*1.732) node {$ 3 \nu_1$};
    \draw (3,2*1.732) node {$a'$};
    \draw (4,2*1.732) node {$b'$};
    \draw (5,2*1.732) node {$c'$};
    \draw (1.5,1.5*1.732) node {$d'$};
    \draw (2.5,1.5*1.732) node {$e'$};
    \draw (3.5,1.5*1.732) node {$f'$};
    \draw (4.5,1.5*1.732) node {$g'$};
    \draw (5.5,1.5*1.732) node {$g'$};
    \draw (1,1.732) node {$h'$};
    \draw (2,1.732) node {$i'$};
    \draw (3,1.732) node {$j'$};
    \draw (4,1.732) node {$k'$};
    \draw (5,1.732) node {$k'$};
    \draw (6,1.732) node {$k'$};
    \draw (0.5,0.5*1.732) node {$l'$};
    \draw (1.5,0.5*1.732) node {$m'$};
    \draw (2.5,0.5*1.732) node {$n'$};
    \draw (3.5,0.5*1.732) node {$o'$};
    \draw (4.5,0.5*1.732) node {$o'$};
    \draw (5.5,0.5*1.732) node {$o'$};
    \draw (6.5,0.5*1.732) node {$o'$};
    \draw[red, thick] (0.75,0.5*1.732) -- (1.25,0.5*1.732);
    \draw[red, thick] (1.75,0.5*1.732) -- (2.25,0.5*1.732);
    \draw[red, thick] (2.75,0.5*1.732) -- (3.25,0.5*1.732);
    \draw (3.75,0.5*1.732) -- (4.25,0.5*1.732);
    \draw (4.75,0.5*1.732) -- (5.25,0.5*1.732);
    \draw (5.75,0.5*1.732) -- (6.25,0.5*1.732);
    \draw (1.25,1.732) -- (1.75,1.732);
    \draw (2.25,1.732) -- (2.75,1.732);
    \draw (3.25,1.732) -- (3.75,1.732);
    \draw (4.25,1.732) -- (4.75,1.732);
    \draw (5.25,1.732) -- (5.75,1.732);
    \draw (1.75,1.5*1.732) -- (2.25,1.5*1.732);
    \draw (2.75,1.5*1.732) -- (3.25,1.5*1.732);
    \draw (3.75,1.5*1.732) -- (4.25,1.5*1.732);
    \draw (4.75,1.5*1.732) -- (5.25,1.5*1.732);
    \draw (2.75,2.5*1.732) -- (3.25,2.5*1.732);
    \draw (3.75,2.5*1.732) -- (4.25,2.5*1.732);
    \draw (3.25,3*1.732) -- (3.75,3*1.732);
    \draw[red, thick] (2.25,2*1.732) -- (2.75,2*1.732);
    \draw[red, thick] (3.25,2*1.732) -- (3.75,2*1.732);
    \draw[red, thick] (4.25,2*1.732) -- (4.75,2*1.732);
    \draw[red, thick] (0.5+0.15, 0.5*1.732+0.2598) -- (0.5+0.35, 0.5*1.732 + 0.6062);
    \draw[red, thick] (1+0.15, 1.732+0.2598) -- (1+0.35, 1.732 + 0.6062);
    \draw[red, thick] (1.5+0.15, 1.5*1.732+0.2598) -- (1.5+0.35, 1.5*1.732 + 0.6062);
    \draw[blue, thick] (2+0.15, 2*1.732+0.2598) -- (2+0.35, 2*1.732 + 0.6062);
    \draw (2.5+0.15, 2.5*1.732+0.2598) -- (2.5+0.35, 2.5*1.732 + 0.6062);
    \draw (3+0.15, 3*1.732+0.2598) -- (3+0.35, 3*1.732 + 0.6062);
    \draw (1.5+0.15, 0.5*1.732+0.2598) -- (1.5+0.35, 0.5*1.732 + 0.6062);
    \draw (2+0.15, 1.732+0.2598) -- (2+0.35, 1.732 + 0.6062);
    \draw[green, thick] (2.5+0.15, 1.5*1.732+0.2598) -- (2.5+0.35, 1.5*1.732 + 0.6062);
    \draw[blue, thick] (3+0.15, 2*1.732+0.2598) -- (3+0.35, 2*1.732 + 0.6062);
    \draw (3.5+0.15, 2.5*1.732+0.2598) -- (3.5+0.35, 2.5*1.732 + 0.6062);
    \draw (2.5+0.15, 0.5*1.732+0.2598) -- (2.5+0.35, 0.5*1.732 + 0.6062);
    \draw (3+0.15, 1.732+0.2598) -- (3+0.35, 1.732 + 0.6062);
    \draw[green, thick] (3.5+0.15, 1.5*1.732+0.2598) -- (3.5+0.35, 1.5*1.732 + 0.6062);
    \draw[blue, thick] (4+0.15, 2*1.732+0.2598) -- (4+0.35, 2*1.732 + 0.6062);
    \draw[red, thick] (3.5+0.15, 0.5*1.732+0.2598) -- (3.5+0.35, 0.5*1.732 + 0.6062);
    \draw[red, thick] (4+0.15, 1.732+0.2598) -- (4+0.35, 1.732 + 0.6062);
    \draw[red, thick] (4.5+0.15, 1.5*1.732+0.2598) -- (4.5+0.35, 1.5*1.732 + 0.6062);
    \draw[green, thick] (4.6+0.15, 1.5*1.732+0.2598) -- (4.6+0.35, 1.5*1.732 + 0.6062);
    \draw (4.5+0.15, 0.5*1.732+0.2598) -- (4.5+0.35, 0.5*1.732 + 0.6062);
    \draw (5+0.15, 1.732+0.2598) -- (5+0.35, 1.732 + 0.6062);
    \draw (5.5+0.15, 0.5*1.732+0.2598) -- (5.5+0.35, 0.5*1.732 + 0.6062);
    
    \draw (1.5-0.15, 0.5*1.732+0.2598) -- (1.5-0.35, 0.5*1.732 + 0.6062);
    \draw (2.5-0.15, 0.5*1.732+0.2598) -- (2.5-0.35, 0.5*1.732 + 0.6062);
    \draw (3.5-0.15, 0.5*1.732+0.2598) -- (3.5-0.35, 0.5*1.732 + 0.6062);
    \draw (4.5-0.15, 0.5*1.732+0.2598) -- (4.5-0.35, 0.5*1.732 + 0.6062);
    \draw (5.5-0.15, 0.5*1.732+0.2598) -- (5.5-0.35, 0.5*1.732 + 0.6062);
    \draw (6.5-0.15, 0.5*1.732+0.2598) -- (6.5-0.35, 0.5*1.732 + 0.6062);
    \draw (2-0.15, 1.732+0.2598) -- (2-0.35, 1.732 + 0.6062);
    \draw (3-0.15, 1.732+0.2598) -- (3-0.35, 1.732 + 0.6062);
    \draw (4-0.15, 1.732+0.2598) -- (4-0.35, 1.732 + 0.6062);
    \draw (5-0.15, 1.732+0.2598) -- (5-0.35, 1.732 + 0.6062);
    \draw (6-0.15, 1.732+0.2598) -- (6-0.35, 1.732 + 0.6062);
    \draw (2.5-0.15, 1.5*1.732+0.2598) -- (2.5-0.35, 1.5*1.732 + 0.6062);
    \draw (3.5-0.15, 1.5*1.732+0.2598) -- (3.5-0.35, 1.5*1.732 + 0.6062);
    \draw (4.5-0.15, 1.5*1.732+0.2598) -- (4.5-0.35, 1.5*1.732 + 0.6062);
    \draw (5.5-0.15, 1.5*1.732+0.2598) -- (5.5-0.35, 1.5*1.732 + 0.6062);
    \draw (3-0.15, 2*1.732+0.2598) -- (3-0.35, 2*1.732 + 0.6062);
    \draw (4-0.15, 2*1.732+0.2598) -- (4-0.35, 2*1.732 + 0.6062);
    \draw (5-0.15, 2*1.732+0.2598) -- (5-0.35, 2*1.732 + 0.6062);
    \draw (3.5-0.15, 2.5*1.732+0.2598) -- (3.5-0.35, 2.5*1.732 + 0.6062);
    \draw (4.5-0.15, 2.5*1.732+0.2598) -- (4.5-0.35, 2.5*1.732 + 0.6062);
    \draw (4-0.15, 3*1.732+0.2598) -- (4-0.35, 3*1.732 + 0.6062);
    \end{tikzpicture}
\caption{Here, $q = 2 \nu_1 + \gamma _1$ and for $x$ a label in $h$ we write $x'$ to denote $x+3\nu_1$}\label{par_to_tri}.
\end{figure}
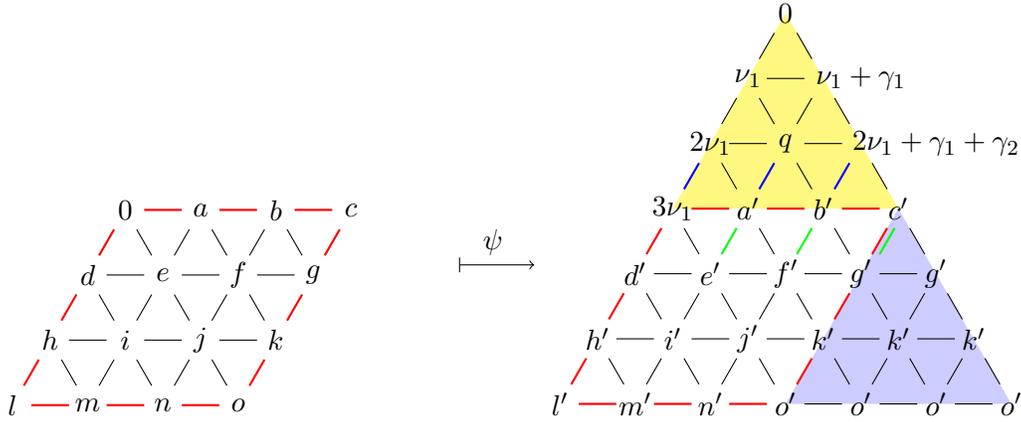

By the above description, to verify the well definedness of the map $\psi$, it suffices to check the rhombus inequalities in $\psi(h)$ for the following $2n$ rhombi:
\begin{itemize}
    \item The $n$ SE rhombi each of which straddles the regions described in (2) and (4).
    \item The $n$ vertical rhombi each of which straddles the regions described in (2) and (3).
\end{itemize}
The first $n$ rhombi inequalities hold because of the fact that the entries of $h$ increase along the rows (this follows form the Southeast rhombi inequalities and the fact that $\mu_t \geq 0 \hspace{0.2cm} \forall t$). 
The second $n$ rhombi inequalities hold because the edge labels \footnote{The edge label of the Northeast edge 
\begin{tikzpicture}
    \draw (0,0) node {$x$};
    \draw (0.5, 0.5*1.732) node {$y$};
    \draw (0+0.15,0.2598) -- (0+0.35,0.6062);
    \end{tikzpicture} 
is defined to be $x-y$.} of the Northeast edges (green edges in figure \ref{par_to_tri}) originating from the $(n+2) ^{th}$ row is bounded above by $\nu_1$, which is equal the edge label of the Northeast edges originating from the $(n+1)^{th}$ row (blue in the figure). This follows from the Northeast rhombi inequalities and the boundary condition on $h$. For example, in figure \ref{par_to_tri} we have $e' - a' \leq f' -b' \leq g'-c' = \nu_1$. Therefore the map $\psi$ is well defined.

It is clear from the definition of the map that it is injective, affine linear and sends integral skew hives to integral hives. We now establish surjectivity. Given a hive in $\Hive(\Tilde{\lambda}, \Tilde{\mu}, \Tilde{\nu})$, consider its triangular subarrays marked in yellow and blue in Figure~\ref{par_to_tri}. These are themselves hives, and lemma~\ref{hive_singleton} implies that these hives are uniquely determined (since the corresponding hive polytopes are non-empty). In particular, the labels on the bottom row of the yellow triangle are  $\bar{\gamma}+n \cdot \nu_1$  where $\bar{\gamma}$ denotes the vector of partial sums of $\gamma$ as defined in \S\ref{sec:skewhive}. Likewise, the labels on the left edge of the blue triangle must be $|\gamma| + \bar{\nu} + n \cdot \nu_1$. These are also edges of the the white parallelogram. The othe two edges of the parallelogram have edge labels $\bar{\lambda} + n \cdot \nu_1$ and $|\lambda| + \bar{\mu} + n \cdot \nu_1$.  This proves surjectivity of $\psi$.

Finally, since the map $\psi$ does not change rhombus contents, it does not alter any flatness conditions within the white parallelogram. Thus if the left and bottom justified region $R(\Phi)$ is flat in $h$, then it remains flat in $\psi(h)$; however since $\psi(h)$ is a triangular hive in twice as many ambient dimensions, this region would now correspond to $R(\Tilde{\Phi})$ in $\psi(h)$, where  $\Tilde{\Phi} = (\Phi_1 +n, \cdots, \Phi_n +n)$.
\end{proof}

We remark that $Hive(\Tilde{\lambda}, \Tilde{\mu}, \Tilde{\nu}, \Tilde{\Phi})$ coincides with the hive Kogan face $K^{\Hive}(\Tilde{\lambda}, \Tilde{\mu}, \Tilde{\nu}, w(\Tilde{\Phi}))$ of \cite{krvfpsac}, where $w(\Tilde{\Phi}) \in \mathfrak{S}_{2n}$ is the unique $312$-avoiding permutation corresponding to the flag $\Tilde{\Phi}$ \cite[\S 14]{PS}, \cite[\S 2.4]{krvfpsac}.
Now, as a consequence of the proposition, we have:
\begin{theorem}
    The flagged skew Littlewood-Richardson coefficients coincide with certain $w$-refined (where, $w$ is $312$-avoiding) Littlewood-Richardson coefficients of \cite{krvfpsac}. More precisely, $c_{\lambda, \,\mu/\gamma} ^{\,\nu} (\Phi) = c_{\Tilde{\lambda}, \,\Tilde{\mu}} ^{\,\Tilde{\nu}} (w(\Tilde{\Phi}))$ where $w(\Tilde{\Phi})$ is the unique $312$-avoiding permutation corresponding to the flag $\Tilde{\Phi}$.
\end{theorem}

It is elementary to check that $\widetilde{k\lambda} = k\Tilde{\lambda}$ and $\widetilde{k\nu} = k\Tilde{\nu}$. This implies, $c_{k\lambda, \,k\mu/k\gamma} ^{\,k\nu} (\Phi) = c_{k\Tilde{\lambda}, \,k\Tilde{\mu}} ^{\,k\Tilde{\nu}} (w(\Tilde{\Phi}))$. But Theorem 1.4 of \cite{krvfpsac} establishes the saturation property of the $w$-refined LR coefficients when $w$ is $312$-avoiding. Together with the preceding remarks, this implies our main theorem:
\begin{theorem}\label{thm:mainthmbody}
The saturation property holds for the flagged LR coefficients. i.e., $$c_{k\lambda, \,k\mu/k\gamma} ^{\,k\nu} (\Phi) >0 \text{ for some } k \geq 1 \implies c_{\lambda, \,\mu/\gamma} ^{\,\nu} (\Phi) >0$$
\end{theorem}
Theorem~\ref{thm:mainthmbody} can also be proved by working directly with the skew hive polytope, rather than with its isomorphic hive polytope. This involves 
 mimicking all the arguments of \cite{buch, krvfpsac} for skew hives. While we have chosen a shorter approach in this paper, this alternate approach naturally suggests numerous other refinements of the LR coefficients with the saturation property. These will be considered in a future publication.

\section{Appendix: Decomposition of $\Tab(\lambda / \mu, \Phi)$ into Demazure crystals}
In this section, we make theorem \ref{crystal-thrm} effective, describing algorithmically the Demazure crystals which occur in the decomposition. The arguments below are implicit in the character level proof of \cite{RS}, and so we content ourselves with sketching their broad contours.

We start with a brief discussion of the Burge correspondence \cite{fulton:yt}. We use the standard notation $[n] =\{1,2,\cdots,n\}$. Consider a matrix $M=(m_{ij})$ of size $r \times n$ with non-negative integer entries. We associate a biword to  $M$  as follows $$ w_{M}=\begin{bmatrix}          i_t\hspace{0.2cm} \cdots \hspace{0.2cm}i_2 \hspace{0.2cm} i_1 \\
           j_t \hspace{0.2cm} \cdots \hspace{0.2cm} j_2 \hspace{0.2cm}j_1 \\
           
\end{bmatrix}$$
such that for any pair $(i, j)$ that indexes an entry $m_{ij}$ of $M$, there are $m_{ij}$ columns equal to
$\begin{bmatrix}
    i \\
    j
\end{bmatrix}$ in $w_M$, and the columns of  $w_M$ are ordered as follows:
\begin{center}
    \begin{enumerate}
        \item $i_t \geq \cdots \geq i_2 \geq i_1 \geq 1 $.
        \item $ i_{k+1} > i_k$ whenever $ j_{k+1}>j_k$.
    \end{enumerate}
\end{center}

In other words, form the biword $w_M$ by reading the entries $m_{ij}$ of $M$ from left to right within each row starting with the bottom row and proceeding upwards, recording each
$\begin{bmatrix}    i \\    j \end{bmatrix}$ with multiplicity $m_{ij}$. 
We will often denote the row and column indices of the biword by $ \textbf{i} =i_1 i_2 \cdots i_t$ and $ \textbf{j} =j_1 j_2 \cdots j_t$. 
Additionally, given a flag $\Phi$, if  $ i_k \leq \Phi_{j_k}$ for all $k$ (in particular, the matrix $M$ is block upper-triangular) then $ \textbf{i}$ is said to be $(\textbf{j},\Phi)$-compatible (see \cite{RS}).
\begin{theorem}
\cite[Appendix A, Proposition 2]{fulton:yt}
 The Burge correspondence gives a bijection between the set of all $r \times n$ matrices with non-negative integer entries $Mat_{r \times n}(\mathbb{Z}_{+})$ and the set of pairs $(P, Q)$ of semistandard tableaux with the same shape where entries of $P$ are in 
   $[n]$ and entries of $Q$ are in 
    $[r]$. We use the notation $(w_{A} \rightarrow \emptyset)=(P, Q)$ if A corresponds to $(P, Q)$.
\end{theorem}
\begin{theorem}
\cite[Appendix A, Symmetry Theorem $($b$)$]{fulton:yt}
\label{burge}
 If  $ M \in Mat_{r \times n}(\mathbb{Z}_{+})$ corresponds to $(P, Q)$ then its transpose ${M^t}$  corresponds to $(Q, P)$.   
\end{theorem}

\emph{The reverse filling} of the skew shape $\lambda/\mu$, denoted $RF(\lambda/\mu)$, is defined to be the filling of the boxes of the shape $\lambda/\mu$ by $1, 2, 3, \cdots, |\lambda/\mu|$, sequentially from right to left within each row, starting with the top row and proceeding downwards.

We say a standard $(\text{skew})$ tableau $Q$ with $|\text{shape$(Q)$}|=|\lambda/\mu|$ is $\lambda/\mu$-\emph{compatible} 
if it satisfies the following:
\begin{enumerate}
    \item If $i+1, i $ are adjacent in a row of $RF(\lambda/\mu)$ then $i+1$ appears weakly north and strictly east of $i$ in $Q$.
    \item If $i$ occurs directly above $j$ in a column of $RF(\lambda/\mu)$ then $j$ appears weakly west and strictly south of $i$ in $Q$.
\end{enumerate}

\begin{remark}
    The set of all $\lambda/\mu$-compatible standard tableaux of shape $\nu$  is in one-to-one correspondence with the set of  Littlewood-Richardson tableaux of shape $\lambda/\mu$ and weight $\nu$ \cite[Chapter 5, Proposition 4]{fulton:yt}. 
\end{remark}

Fix a $\lambda/\mu$-compatible standard tableau $Q$. Given a composition $\alpha = (\alpha_1, \alpha_2, \cdots)$, let \textbf{b}($\alpha$) be the word $ \cdots b^{(2)}b^{(1)} $ in which $b^{(j)}$ consists of a string of $\alpha_j$ copies of $j$. Also, consider the semi-standard tableau $R$ whose standardization\footnote{The standardization of a tableau $T$ (denoted by std$(T)$) is the tableau obtained by changing the $1$'s in $T$ from left to right to $1,2,\cdots, \alpha_1$, then the $2$'s to $\alpha_1+1,\alpha_1+2,\cdots,\alpha_1+\alpha_2$ etc, where $\alpha=$weight$(T)$.} is $Q$ and weight is the composition $\rho=\lambda-\mu$. If  $\textbf{i}=i_1i_2\cdots i_t $ then by $\text{rev}(\textbf{i})$ we mean the word $i_t\cdots i_2i_1$.\\
Define $\mathcal{A}(Q,\lambda/\mu,\Phi)=\{ T \in \Tab(\lambda/\mu, \Phi):(
\begin{bmatrix}
           \textbf{b}(\rho) \\
           \text{rev}(b_T) \\
           
\end{bmatrix} \rightarrow \emptyset)
$  
$=(-, R)\}$.
Then $ \Tab(\lambda/\mu,\Phi)= \bigsqcup \mathcal{A}(Q,\lambda/\mu,\Phi) $ where the union is over all $\lambda/\mu$-compatible standard tableaux $Q$ \cite{RS}. We will show that $\mathcal{A}(Q,\lambda/\mu,\Phi)$ is isomorphic to some Demazure crystal as crystals, i.e., there is a weight-preserving bijection between these sets which intertwines the crystal raising and lowering operators (where defined).

For a composition $\alpha$, key($\alpha$) is the semi-standard tableau of shape $\alpha^{\dagger}$ whose first $\alpha_k$ columns contain the letter $k$ for all $k$. One can see that key($\alpha$) is the unique tableau of shape $\alpha ^{\dagger}$ and weight $\alpha$. We define $\mathcal{W}(\alpha, \Phi)$ as the set of all words $\textbf{u} = \cdots u^2 \cdot u^1 $, where each $u^i$ is a maximal row word  of length  $\alpha_i$ together with the properties that each letter in $u^i$ can be atmost $\Phi_i$ and $(
\begin{bmatrix}
    \textbf{b}(\alpha)\\
    \textbf{u}
    \end{bmatrix} \rightarrow \emptyset) =(-, \text{key}(\alpha)) $.

\medskip
Let $\Phi_0=(1,2,3,\cdots,n)$ denote the {\em standard flag}. We have:
\begin{theorem}
    \cite[Proposition 5.6]{LS-keys} Let $\alpha = \omega \alpha ^{\dagger}$. Then the set $\mathcal{W}(\alpha, \Phi_0)$ has a one-to-one correspondence with the set $\mathcal{B}_{\omega}(\alpha^{\dagger})$ via $\textbf{u} \mapsto P(\textbf{u})$ where $ P(\textbf{u})$ is the unique tableau that is Knuth equivalent to $\textbf{u} $.
\end{theorem}

Let $\textbf{a}=a_1a_2\cdots a_t$ be a word. Then by an \emph{ascent} of the word \textbf{a} we mean a positive integer $1\leq k \leq t-1$ such that $a_k < a_{k+1}$. We recursively define the \emph{essential subword} $\text{ess}_L(\textbf{a})$ of \textbf{a}  with respect to a positive integer $L$ or $\infty$ to be the following indexed subword of \textbf{a}:
\begin{enumerate}
    \item $\text{ess}_L(\textbf{a})$ is the empty word if $t=0$.
    \item $\text{ess}_L(\textbf{a})=\text{ess}_{a_t}(a_1a_2 \cdots a_{t-1})\,a_t$ if $a_t < L $.
    \item  $ \text{ess}_L(\textbf{a}) =\text{ess}_L(a_1a_2 \cdots a_{t-1})$ if $a_t\geq L $.
\end{enumerate}
Define the essential subword of \textbf{a} as $\text{ess}(\textbf{a})=\text{ess}_{\infty}(\textbf{a})$.
\begin{lemma}
\label{lemma 8}
    Let $\Phi $ be a flag and $\textbf{i}=i_1 i_2\cdots i_t$ be a word. If $\textbf{a}=a_1 a_2\cdots a_t$, $\textbf{b}=b_1 b_2 \cdots b_t$  are words in $[n]$ having the same essential subword and ascents, then \textbf{i} is $(\textbf{a}, \Phi) $-compatible if and only if \textbf{i} is $(\textbf{b}, \Phi)$-compatible.
\end{lemma}
\begin{proof}
 The   proof is similar to that of Lemma 8 of \cite{RS}.
\end{proof}
For a semi-standard tableau $T$, let $T|_{<L}$ denote the subtableau of $T$ consisting of the entries of $T$ which are less than $L$, and let $ K_{\_}(T)$ denote the left key tableau of $T$. For more details on computing left and right keys, see \cite{RS}, \cite{willis} and \cite{plactic}.
\begin{lemma}
\label{lemma 9}
  Let $L\geq1$. Suppose that 
  \begin{enumerate}
      \item  $P, P'$ and $Q$ are semi-standard tableaux of the same shape such that $wt(Q) = (1,1, \cdots, 1)$ and 
      $ K_{\_}(P)|_{<L} = K_{\_}(P')|_{<L} $. 
      \item  $\textbf{a}=a_1a_2 \cdots a_t,$ $ \textbf{a}'=a'_1a'_2 \cdots a'_t 
      \text{ are two words such that } 
      (
\begin{bmatrix}
    \textbf{b}(\textbf{1}_t)\\
    \textup{rev}(\textbf{a})
    \end{bmatrix} \rightarrow \emptyset)
      =(P, Q)$ and 
      $ (
\begin{bmatrix}
    \textbf{b}(\textbf{1}_t)\\
    \textup{rev}(\textbf{a}')
    \end{bmatrix} \rightarrow \emptyset)
      =(P', Q)$, where $\textbf{1}_t$ is the composition $(1,1,\cdots,1)$ of length $t$.
  \end{enumerate}  
  Then $\textup{ess}_L(\textbf{a}) = \textup{ess}_L(\textbf{a}')$ and $ \textbf{a},\textbf{a}'$ have the same ascents.
\end{lemma}
\begin{proof}
 The proof is similar to that of Lemma 9 of \cite{RS}.
\end{proof}

Now we have the following proposition:
\begin{proposition}
There is a bijection $\Omega$ between the sets $\mathcal{A}(Q,\lambda/\mu,\Phi)$ and $ \mathcal{W}(\beta(R), \Phi)$  such that if $ T \mapsto \Omega (T)$ then $\textup{rev}(b_T)$ and $\Omega{(T)}$ are Knuth equivalent. Here $\beta(R)$ denote the weight of the left key tableau $K_{\_}(R)$ of $R$.   
\end{proposition}
\begin{proof}
    Let $T$  $ \in \Tab(\lambda/\mu, \Phi)$ and $M(T)$ be the matrix corresponding to
    $
\begin{bmatrix}
           \textbf{b}(\rho) \\
           \text{rev}(b_T) \\
           
\end{bmatrix}$.
Suppose that $
\begin{bmatrix}
           \text{rev}(\textbf{i}) \\
           \text{rev}(\textbf{a}) \\
           
\end{bmatrix}$ is the biword for $M(T)^t$. Thus \textbf{i} is $(\textbf{a}, \Phi)$-compatible because $T \in \Tab(\lambda/\mu, \Phi)$.\\ 

Let $\text{rect}(T)$ denote  the rectification of the skew tableau $T$.
Then $(
\begin{bmatrix}
           \textbf{b}(\rho) \\
           \text{rev}(b_T) \\
           
\end{bmatrix} \rightarrow \emptyset)
=(\text{rect}(T), R) \implies 
$
$(
\begin{bmatrix}
           \text{rev}(\textbf{i}) \\
           \text{rev}(\textbf{a}) \\
           
\end{bmatrix} \rightarrow \emptyset)=(R,\text{rect}(T))
$ (by theorem \ref{burge}).
Consider the unique word $\textbf{a}'$ such that 
$(
\begin{bmatrix}
           \text{rev}(\textbf{i}) \\
           \text{rev}(\textbf{a}') \\
           
\end{bmatrix} \rightarrow \emptyset)=(K_{\_}(R),\text{rect}(T))$. 
Then by Lemma~\ref{lemma 8} and Lemma~\ref{lemma 9},  \textbf{i} is $(\textbf{a}',\Phi) $-compatible. Let $ \begin{bmatrix}
           \text{rev}(\textbf{j}) \\
           \text{rev}(\textbf{v}) \\
           
\end{bmatrix} $
be the biword associated to the matrix A such that $ A^t$ corresponds to
$
\begin{bmatrix}
           \text{rev}(\textbf{i}) \\
           \text{rev}(\textbf{a}') \\
           
\end{bmatrix}
$. Hence $(
\begin{bmatrix}
           \text{rev}(\textbf{j}) \\
           \text{rev}(\textbf{v}) \\
           
\end{bmatrix} \rightarrow \emptyset)=(\text{rect}(T), K_{\_}(R))
$ (by theorem \ref{burge}).
So by corollary 12 of \cite{RS}, we have $\text{rev}(\textbf{v}) \in  \mathcal{W}(\beta(R), \Phi).$
We define $\Omega(T) = \text{rev}(\textbf{v})$. Then 
$\Omega$ is a bijection and $\text{rev}(b_T)$ and $\Omega(T) $ are Knuth equivalent.
\end{proof}
\begin{theorem}
    \cite[Theorem 21]{RS}
    For a flag $\Phi$ and a composition $\beta$, either $\mathcal{W}(\beta, \Phi)$ is empty or there is a bijection $\zeta $ between the sets $\mathcal{W}(\beta, \Phi)$ and $ \mathcal{W}(\hat{\beta}, \Phi_0)$ for some composition $\hat{\beta}$ with $\beta^{\dagger} = \hat{\beta}^{\dagger}$ such that if $\textbf{u} \mapsto \zeta (\textbf{u})$ then $\textbf{u}$ and $\zeta (\textbf{u}) $ are Knuth equivalent.
\end{theorem}

Now the following proposition tells us that $ \Tab(\lambda/\mu, \Phi)$ is a disjoint union of Demazure crystals.
\begin{proposition}
The rectification map $ \textup{rect}: \mathcal{A}(Q,\lambda/\mu,\Phi) \rightarrow  \mathcal{B}_{\tau}(\widehat{\beta(R)}^{\dagger})$ is a weight-preserving bijection which intertwines the crystal raising and lowering operators $($where defined$)$. Here $\tau$ is any permutation such that $\tau. \widehat{\beta(R)}^{\dagger}=\widehat{\beta(R)}$.
\end{proposition}

\begin{proof}
  We know that $\text{rev}(b_T),\Omega(T), \zeta(\Omega(T)),  P(\zeta(\Omega(T)))$ are all Knuth equivalent. Thus we have $\text{rect}(T)$ $ = P(\zeta(\Omega(T))) \in \mathcal{B}_{\tau}(\widehat{\beta(R)}^{\dagger})$. So the map is well-defined. Clearly, the rectification map is a weight-preserving bijection. Commutativity of the rectification map with the crystal raising and lowering operators  comes from properties of Knuth equivalence. 
\end{proof}

\printbibliography
\end{document}